\documentclass[a4paper]{amsart}

\usepackage[utf8]{inputenc}
\usepackage[english]{babel}
\usepackage{tikz-cd}
\usetikzlibrary{positioning,decorations.text,quotes,calc}
\usetikzlibrary{arrows,babel}
\usepackage{amsmath,amsfonts,amssymb}
\usepackage{mathtools}
\usepackage{float}
\usepackage{amsthm}
\usepackage{stackrel}
\usepackage{pgfplots}
\usepackage{graphicx}
\usepackage{bbm}
\usepackage{amssymb}
\usepackage{geometry,enumerate}
\usepackage{graphicx}
\usepackage{psfrag}
\usepackage{amscd}
\usepackage{comment}
\usepackage[all]{xy}
\usepackage{url}
\usepackage{leftindex}
\usepackage{mathrsfs}

\usepackage{fullpage}

\usepackage[dvipsnames]{xcolor}

\usepackage[backref=page]{hyperref}

\hypersetup{
 colorlinks,
 citecolor=Green,
 linkcolor=blue,
 urlcolor=Blue}

\usepackage{booktabs}
\usepackage{color}
\usepackage{todonotes}

\usepackage{braket}

\usepackage{stmaryrd} 
\usepackage{mathabx}
\usepackage{mathbbol}
\usepackage{xspace}
\usepackage{enumerate}
\usepackage{caption}
\usepackage{subcaption}

 \usepackage{cleveref}
\crefformat{section}{#2#1#3} 
\crefformat{subsection}{#2#1#3}

\newcommand{\ZZ}{\mathbb{Z}}
\newcommand{\QQ}{\mathbb{Q}}
\newcommand{\NN}{\mathbb{N}}
\newcommand{\RR}{\mathbb{R}}

\newcommand{\A}{\mathcal{A}}
\newcommand{\C}{\mathcal{C}}
\renewcommand{\O}{\mathcal{O}}

\renewcommand{\L}{\mathcal{L}}

\newcommand{\I}{\mathcal{I}}
\renewcommand{\L}{\mathcal{L}}

\newcommand{\F}{\mathcal{F}}
\newcommand{\D}{\mathcal{D}}

\newcommand{\J}{\mathcal{J}}

\newcommand{\M}{\mathcal{M}}
\newcommand{\R}{\mathcal{R}}

\newcommand{\un}{\underline}
\newcommand{\ov}{\overline}
\newcommand{\wh}{\widehat}
\newcommand{\wt}{\widetilde}
\newcommand{\cal}{\mathcal}

\newcommand{\Mbargn}{\overline{\mathcal{M}}_{g,n}}
\newcommand{\Cbargn}{\overline{\mathcal{C}}_{g,n}}

\newcommand{\Mgn}{{\mathcal{M}}_{g,n}}
\newcommand{\Cgn}{{\mathcal{C}}_{g,n}}
\newcommand{\Bgn}{\mathbb{B}_{g,n}}

\newcommand{\Mbarg}{\overline{\mathcal{M}}_{g}}

\newcommand{\Sigmas}{\operatorname{^{s}\Sigma}}
\newcommand{\Sigmans}{\operatorname{^{ns}\Sigma}}

\newcommand{\Dgn}{\mathbb{D}_{g,n}}
\newcommand{\Dgns}{{^{s}\mathbb{D}_{g,n}}}
\newcommand{\Dgnns}{{^{ns}\mathbb{D}_{g,n}}}

\newcommand{\wPR}{\wt{\operatorname{PR}}}
\newcommand{\PR}{\operatorname{PR}}

\newcommand{\Proj}{\operatorname{Proj}}

\newcommand{\coker}{\operatorname{coker}}

\newcommand{\Hom}{\operatorname{Hom}}
\newcommand{\wei}{\operatorname{wt}}
\newcommand{\Amp}{\operatorname{Amp}}
\newcommand{\irr}{\operatorname{irr}}

\newcommand{\rk}{\operatorname{rk}}
\newcommand{\BCon}{\operatorname{BCon}}

\newcommand{\Pic}{\operatorname{Pic}}
\renewcommand{\Im}{\operatorname{Im}}

\newcommand{\Spec}{\operatorname{Spec}}

\newcommand{\Tors}{\operatorname{Tors}}

\newcommand{\type}{\operatorname{type}}
\newcommand{\RelPic}{\operatorname{RelPic}}

\newcommand{\res}{\operatorname{res}}
\newcommand{\NS}{\operatorname{NS}}

\newcommand{\TF}{\operatorname{TF}}

\newcommand{\Gm}{\operatorname{\mathbb{G}_m}}

\newcommand{\Z}{\mathbb{Z}}

\pgfplotsset{compat=1.17}

\definecolor{LivGreen}{cmyk}{68 0 100 0}

\newtheorem{theorem}{Theorem}[section]
\newtheorem{corollary}[theorem]{Corollary}

\newtheorem{proposition}[theorem]{Proposition}
\newtheorem{proposition-definition}[theorem]{Proposition-Definition}
\newtheorem{lemma-definition}[theorem]{Lemma-Definition}

\newtheorem{theoremalpha}{Theorem}

\newtheorem*{corollary*}{Corollary}

\theoremstyle{definition}
\newtheorem{definition}[theorem]{Definition}

\newtheorem{remark}[theorem]{Remark}

\numberwithin{equation}{section}

\newenvironment{sis}{\left\{\begin{aligned}}{\end{aligned}\right.}

\usepackage{xpatch}
\let\sectionnotoc\section
\xpatchcmd{\sectionnotoc}{{1}}{{1001}}{}{}
\let\subsectionnotoc\subsection
\xpatchcmd{\subsectionnotoc}{{2}}{{1001}}{}{}

\begin{document}

\title{On the projectivity of compactified universal Jacobians}

\author{Filippo Viviani}
\address{Filippo Viviani, Dipartimento di Matematica, Universit\`a di Roma ``Tor Vergata'', Via della Ricerca Scientifica 1, 00133 Roma, Italy}
\email{viviani@mat.uniroma2.it}



\begin{abstract}
Compactified universal Jacobian stacks/spaces over $\Mbargn$ were classified and studied in \cite{FPV3} (based upon the previous works \cite{FPV1} and \cite{FPV2}).

The aim of this paper is to investigate the projectivity of compactified universal Jacobian spaces, a question which was left open in \cite{FPV3}.

We show that the relatively projective compactified universal Jacobian spaces are exactly the classical ones (constructed independently, in the fine case, by Kass-Pagani \cite{Kass_2019} and Esteves-Melo \cite{esteves, Melo15, melo2019}), i.e those that come from (relative) $\mathbb{R}$-line bundles on the universal curve $\Cbargn/\Mbargn$ or equivalently from  vector bundles on the universal curve $\Cbargn/\Mbargn$.

A crucial step in the proof is the computation of the relative Picard groups of any compactified universal Jacobian stack/space.
\end{abstract}

\maketitle
	\bigskip
	
	\tableofcontents

\section{Introduction}    

For any pair of integers $(g,n)\in \NN^2$ such that $2g-2+n>0$, the \emph{universal Jacobian stack} of type $(g,n)$ is the stack $\J_{g,n}$ parametrizing pairs $(C,p_1, \ldots, p_n,L)$, where $(C,p_1, \ldots, p_n)$ is an element of the stack $\M_{g,n}$ of $n$-pointed smooth projective connected curves of genus $g$, and $L$ is a line bundle on $C$. The stack $\J_{g,n}$ admits a decomposition into a disjoint union of irreducible stacks $\coprod_{\chi \in \Z}\J_{g,n}^\chi$, where $\J_{g,n}^\chi$ parametrizes line bundles of Euler characteristic $\chi$. We will denote by $J_{g,n}:=\J_{g,n}\fatslash \Gm$, that is, the rigidification of $\J_{g,n}$ by the group $\Gm$ of scalar automorphisms (and similarly $J_{g,n}^\chi:=\J_{g,n}^\chi\fatslash \Gm$) and call $J_{g,n}$ the \emph{universal Jacobian space}. 

%The relative good moduli space of $\J_{g,n}\to \M_{g,n}$ is the universal Jacobian space $J_{g,n}=\coprod_{\chi\in \Z}J_{g,n}^\chi$ and the morphism $\J_{g,n}\to J_{g,n}$ is the $\Gm$-rigidification. 

\vspace{0.1cm}

There is a natural way of compactifying in a modular way the universal Jacobian stack/space over the moduli stack $\Mbargn$ of stable $n$-pointed nodal curves of genus $g$, as we now recall. 

   A \textbf{compactified universal Jacobian stack} of characteristic $\chi\in \Z$ over $\Mbargn$ is an open substack $\ov \J_{g,n}^{\chi}$ of the stack $\TF^{\chi}_{g,n}$ of relative rank-$1$ torsion-free sheaves of characteristic $\chi$ on $\Cbargn/\Mbargn$, admitting a relative proper good moduli space $\ov J_{g,n}^\chi\rightarrow \Mbargn$, called a \textbf{compactified universal Jacobian space}.  A compactified universal Jacobian stack $\ov \J_{g,n}^\chi$ (resp. space $\ov J_{g,n}^\chi$) is called \emph{fine} if $\ov J_{g,n}^{\chi}=\ov \J_{g,n}^\chi\fatslash \Gm$.
   The set of compactified universal Jacobian stacks over $\Mbargn$ is a poset under the natural inclusion relation. Note that any such compactified universal Jacobian stack $\ov \J_{g,n}^\chi$ (resp. space $\ov J_{g,n}^\chi$) is a "compactification" of the universal Jacobian stack (resp. space) since the restriction of $\ov \J_{g,n}^\chi$ (resp. of $\ov J_{g,n}^\chi$) to $\M_{g,n}$ is $\mathcal{J}_{g,n}^\chi$ (resp. $J_{g,n}^\chi$).

\subsectionnotoc{History and applications}

The search for compactified universal Jacobians (stacks and spaces) has been carried out by many authors using different approaches over the last thirty years. The first such compactification is due to Caporaso \cite{Cap94}, who constructed a compactified universal Jacobian space over $\ov\M_g$ using GIT of suitable Hilbert schemes of curves (see also \cite{pandharipande1995compactification} for another construction of this space using slope semistability). The corresponding compactified universal Jacobian stack over $\ov\M_g$ has been later studied by Caporaso \cite{Cap08} and Melo \cite{melo2009}, and it has been extended to $\Mbargn$ by Melo \cite{melo2011}. Later on, plenty of fine compactified universal Jacobians have been constructed over $\Mbargn$, independently, by Kass-Pagani \cite{Kass_2019} (building upon the work of Oda-Seshadri \cite{Oda1979CompactificationsOT} and Simpson \cite{simpson}) and by Melo \cite{Melo15, Melo17, melo2019} (building upon work of Esteves \cite{esteves}). Indeed, the two constructions produce the same set of fine compactified universal Jacobians  (see \cite[Rmk. 4.6]{Kass_2019} and \cite[Prop.~4.17]{melo2019}), that we call \emph{classical} fine compactified universal Jacobians, see subsection \ref{Sub:classcJ}. 

%Many compactified universal Jacobians have been constructed during the last thirty years by many authors using different techniques, see \cite{Oda1979CompactificationsOT}, \cite{Cap94}, \cite{pandharipande1995compactification}, \cite{esteves}, \cite{melo2009}, \cite{melo2011}, \cite{Melo15}, \cite{Melo17}, \cite{melo2019}, \cite{Kass_2019}. 

\vspace{0.1cm}

Compactified universal Jacobians have found several applications, among which we mention: a modular extension of the Torelli map (see \cite{alexeev}) and a Torelli theorem for stable curves (see \cite{caporasoviviani}); the study of the birational geometry (e.g. Kodaira dimension and Iitaka fibration) of $J_g$ (see \cite{BFV} and \cite{CMKV3}); the study of the tropicalization of $\J_{g,n}$ (see \cite{AP20}, \cite{MMUV}, \cite{AAPT}), its relation to the logarithmic universal Jacobian (see \cite{MW}, \cite{MMUVW}); its relation with the double ramification cycles (see \cite{dudin}, \cite{KPH}) and with the logarithmic double ramification cycles (see \cite{Mol23}, \cite{HMPPS}, \cite{MolchoFourier}); the computation and wall-crossing phenomena for universal Brill-Noether classes (see \cite{Kass_2017}, \cite{PRvZ}, and \cite{APag}). Recently, the cohomology of compactified universal Jacobians has also been the subject of investigation (see \cite{Yin}, \cite{wood}, \cite{Maulikcohom}, \cite{PPSW}, \cite{Wu}).

\subsectionnotoc{Classification of compactified universal Jacobians}

Recently, compactified universal Jacobians over $\Mbargn$ have been classified  by Fava-Pagani-Viviani in \cite{FPV3}, based on their previous works \cite{FPV1} and \cite{FPV2} (see also \cite{pagani2023stability}, \cite{viviani2023new} and \cite{fava2024} for the fine case).   We review this classification in detail in Section \ref{Sec:cuJ}, limiting ourselves to a brief summary in this introduction. 

Given a curve $X=(X,p_i)\in \Mbargn$, we look at the set $\BCon(X)$ of biconnected subcurves of $X$, i.e. subcurves $Y$ of $X$ such that both $Y$ and its complementary subcurve $Y^c:=\ov{X-Y}$ are connected, which comes with the following type function 
\begin{equation*}
    \begin{aligned}
      \type_X: \BCon(X) & \longrightarrow \Dgn:=\left\{ \begin{aligned}
          (e;h,A): 1\leq e, 0\leq h\leq g-e+1, A\subseteq [n], \\
          0<2h-2+e+|A|<2g-2+n
          \end{aligned}\right\}\\
      Y & \mapsto \type_X(Y):=(|Y\cap Y^c|; g(Y),\{i\in [n]\: : p_i\in Y\}).
    \end{aligned}
\end{equation*}
We then define the set $\Sigma_{g,n}$ of \emph{V-functions} of type $(g,n)$, which are certain functions 
 \begin{align*}
        \sigma:\Dgn&\to \ZZ
    \end{align*}
satisfying two conditions with respect to complements and triangles in $\Dgn$ (see Definition \ref{D:Sigmagn}), endowed with a poset structure coming from the pointwise order relation.

Then the classification result of \cite{FPV3} (see Theorem \ref{T:class-cJ} and the references therein) says that  there is an anti-isomorphism of posets  
$$
\begin{aligned}
\Sigma_{g,n} & \xrightarrow{\cong} \left\{\text{Compactified universal Jacobian stacks over $\Mbargn$}\right\},\\
\sigma &\mapsto \ov \J_{g,n}(\sigma):=\left\{
\begin{aligned} 
& (X,I) \in  \TF_{g,n}^{|\sigma|}: \: \chi(I_Y)\geq \sigma(\type_X(Y)) \\ 
&  \text{ for any } Y\in \BCon(X) 
\end{aligned} 
\right\},
\end{aligned}
$$
where $\displaystyle I_Y:=\frac{I_{|Y}}{\Tors(I_{|Y})}$ is the torsion-free quotient of the restriction $I_{|Y}$  (it is called the torsion-free restriction of $I$ to $Y$) and $|\sigma|\in \ZZ$ is the characteristic of $\sigma$. 

\subsectionnotoc{Classical compactified universal Jacobians}
A special class of compactified universal Jacobian was studied in \cite[Sec. 4]{FPV3} and called \textbf{classical} because they include all compactified universal Jacobian stacks/spaces constructed in the literature prior to the full classification obtained in \cite{FPV3} (and to the examples appearing in \cite{PTgenus1}, \cite{pagani2023stability}, \cite{fava2024} in the fine case).

Let us review the definition of classical compactified universal Jacobians (see Subsection \ref{Sub:classcJ} for more details). Denote by $\pi:\Cbargn\to \Mbargn$ the universal family. There are two ways of defining classical compactified universal Jacobian stacks:
\begin{itemize}
    \item[(A)] For any real line bundle $L$ on $\Cbargn$ of integral relative degree $\deg_{\pi}(L)=\chi$,
 we define 
$$
\ov \J_{g,n}(L):=
   \left\{
\begin{aligned} 
& (X,I) \in  \TF_{g,n}^{\chi}: \: \chi(I_Y)\geq \deg_Y(L_{|X}) \\ 
&  \text{ for any } Y\in \BCon(X)
\end{aligned} 
\right\}
$$
and we denote by $\ov J_{g,n}(L)$ the associated space.
\item[(B)] For any vector bundle $E$ on $\Cbargn$ with integral relative slope $\mu_{\pi}(E):=\frac{\deg_{\pi}(E)}{\rk E}=-\chi$, we define 
$$
\ov \J_{g,n}^E:=
   \left\{
\begin{aligned} 
& (X,I) \in  \TF_{g,n}^{\chi}: \: \chi(I_Y)\geq -\frac{\deg_Y(E_{|X})}{\rk E} \\ 
&  \text{ for any } Y\in \BCon(X)
\end{aligned} 
\right\},
$$
and we denote by  $\ov J_{g,n}^E$ the associated space.
\end{itemize}
In the fine case, the compactified universal Jacobians $\{\ov \J_{g,n}(L)\}$ were defined and studied by Kass-Pagani \cite{Kass_2019} (building upon the work of Oda-Seshadri \cite{Oda1979CompactificationsOT} and Simpson \cite{simpson}), while the compactified universal Jacobians $\{\ov \J_{g,n}^E\}$ were defined and studied by Melo \cite{Melo15, Melo17, melo2019}  (building upon work of Esteves \cite{esteves}). Indeed, the above authors proved that the two constructions produce the same set of fine compactified universal Jacobians  (see \cite[Rmk. 4.6]{Kass_2019} and \cite[Prop.~4.17]{Melo15}).

The extension of the above two constructions to the non-fine case is due to Fava-Pagani-Viviani (see \cite[Ex. 4.27, 6.7]{FPV1}) and the fact that they produce, also in the non-fine case, the same set of compactified universal Jacobians follows from the results of \cite[Sec. 4]{FPV3}. 

Each of the above two constructions has its own advantages, as we now point out.

\begin{itemize}
    \item The construction (A) allows for a classification of classical compactified universal Jacobians. Indeed, consider the real affine space 
$$
\RelPic^\chi(\Cbargn)_{\RR}:=\frac{\{L\in \Pic(\Cbargn)_{\RR}\: : \deg_{\pi}(L)=\chi\}}{\pi^*\Pic(\Mbargn)_{\RR}},
$$
which can be explicitly described as a consequence of the weak Franchetta's conjecture (see Theorem \ref{T:RPic-Cbar}). There is an explicit hyperplane arrangement in $\RelPic^{\chi}(\Cbargn)_{\RR}$ (see \eqref{E:arr-hyp}) whose poset of regions $\R_{g,n}^\chi$ is such that  $\ov \J_{g,n}(L)=\ov \J_{g,n}(L')$ if and only if $L$ and $L'$ belong to the same region, which we denote by $[L]=[L']$. Hence, there is an anti-isomorphism of posets (see \eqref{E:sigma-pos} and Definition \ref{D:classcJ})
$$
\begin{aligned}
  \R_{g,n}=\bigsqcup_{\chi \in \ZZ}\R_{g,n}^\chi & \xrightarrow{\cong} \{\text{Classical compactified universal Jacobian stacks}\}\\
  [L]& \mapsto \ov J_{g,n}([L]).
\end{aligned}
$$
    \item The construction (B) allows us to describe explicitly a relative polarization on a classical compactified universal Jacobian space, i.e. a line bundle which is relatively ample over $\Mbargn$. More precisely, a polarization for $\ov J_{g,n}^E$ is the descent of the following line bundle on  $\ov \J_{g,n}^E$ (which we call the \emph{Esteves's line bundle}):
    \begin{equation*}
\L_{E}:=d_{\pi}(\I\otimes p_1^*(E))^{-1},
\end{equation*}
where $d_{\pi}$ denotes the determinant of cohomology with respect to the universal family $\pi:\Cbargn\times_{\Mbargn} \ov \J_{g,n}^E\to \ov \J_{g,n}^E$, $\I$ is the universal sheaf and $p_1$ is the projection onto the first factor (see \cite[Thm. 7.1]{FPV1} which is based upon \cite[Sec. 6,7]{esteves}).
\end{itemize}

\subsectionnotoc{Our results}

A compactified universal Jacobian space $\ov J_{g,n}(\sigma)$ is called \textbf{relatively projective} if the morphism $\ov J_{g,n}(\sigma)\to \Mbargn$ is projective (rather than just proper); and a compactified universal Jacobian stack $\ov \J_{g,n}(\sigma)$ is called relatively projective if its associated compactified universal Jacobian space $\ov J_{g,n}(\sigma)$ is relatively projective. Equivalently, a compactified universal Jacobian space is relatively projective if and only if its coarse moduli space is projective (over $\Spec(\Z)$), see the proof of Corollary \ref{C:proj-coarse}. 

An interesting and natural problem, which was left open in \cite{FPV3}, is to determine whether or not all compactified universal Jacobian stacks/spaces are relatively projective. The goal of this paper is to answer negatively to this question. 

More generally, we classify all the compactified universal Jacobian stacks/spaces which are relatively projective and we show that they coincide with the classical compactified universal Jacobian stacks/spaces.

\begin{theoremalpha}(see Theorem \ref{T:cla-proj})\label{T:thmA}
\noindent 
\begin{enumerate}
    \item \label{T:thmA1} A compactified universal Jacobian space is relatively projective  if and only if it is isomorphic over $\overline{\mathcal{M}}_{g,n}$ to a classical compactified universal Jacobian space.
    %$\overline{J}_{g,n}([L])$, for some region $[L]\in \mathcal{R}_{g,n}$.
    \item \label{T:thmA2} A compactified universal Jacobian stack  is relatively projective  if and only if it is equal to a classical compactified universal Jacobian stack.
\end{enumerate}
\end{theoremalpha}
Since there are non-classical compactified universal Jacobians (also fine ones) as soon as $g,n>0$ and $2g+n \geq 8$ (see Subsection \ref{Sub:classcJ}), the above Theorem implies that there are non-relatively projective compactified universal Jacobian stacks/spaces, which answers negatively the Open Question (1) of \cite{FPV3}.

\vspace{0.1cm}

A crucial ingredient in the proof of the above result is the computation of the relative Picard groups of the universal Jacobian stack/space and of any compactified universal Jacobian stack/space, a result that is interesting in its own.

\begin{theoremalpha}\label{T:thmB}(see Theorems \ref{T:RPicJ}, \ref{T:RPicJ-rig}, \ref{T:RelPic-cJ} and Corollary \ref{C:amplecJ}) 
\noindent 
\begin{enumerate}
\item \label{T:thmB1} The relative Picard group $\RelPic(\J_{g,n}^\chi/\Mgn)$ is generated by
$$
d(\L), \quad \langle \L, \L \rangle, \quad \xi_i:=\langle \L, \O(\sigma_i)\rangle=\sigma_i^*(\L) \text{ for } 1\leq i \leq n,
$$
where $\L$ is the tautological line bundle on the universal family $\pi:\Cgn\times_{\Mgn} \J_{g,n}^\chi\to \J_{g,n}^\chi$, subject to the following relations:
\begin{itemize}
    \item if $g=1$ then $\langle \L, \L \rangle\equiv d(\L)^2$; 
    \item if $g=0$ then $\xi:=\xi_1=\ldots \xi_n$, $d(\L)=\chi \xi$ and $\langle \L, \L\rangle= 2(\chi-1)\xi$.
    \end{itemize}
    \item \label{T:thmB2} The relative Picard group 
$\RelPic(J_{g,n}^\chi/\Mgn)$  is the subgroup of  $\RelPic(\J_{g,n}^\chi/\Mgn)$ given by 
$$\RelPic(J_{g,n}^\chi)=\left\{
(2s-r)d(\L)-s\langle \L,\L\rangle-\sum_i a_i \xi_i \: : 
r\chi+s(2g-2)+\sum_i a_i=0
\right\}$$
and the relative ample cone of  $J_{g,n}^\chi\to \Mgn$ is equal to 
$$\Amp(J_{g,n}^\chi/\Mgn)=\left\{
(2s-r)d(\L)-s\langle \L,\L\rangle-\sum_i a_i \xi_i \: : 
r\chi+s(2g-2)+\sum_i a_i=0, r>0
\right\}.$$
    \item \label{T:thmB3}
Let $\ov \J_{g,n}(\sigma)$ be a compactified universal Jacobian stack and let $\ov J_{g,n}(\sigma)$ its associated compactified universal Jacobian space. Then we have the following commutative diagram
\begin{equation*}
\begin{tikzcd}
  \res: \RelPic(\ov \J_{g,n}(\sigma)/\Mbargn)  \arrow["\res^{\leq 1}", r, "\cong" swap] & \RelPic(\ov \J_{g,n}(\sigma)^{\leq 1}/\Mbargn^{\leq 1}) \arrow[r, "\res^o", twoheadrightarrow]& \RelPic(\J_{g,n}^\chi/\Mgn)   \\ 
  \wh{\res}: \RelPic(\ov J_{g,n}(\sigma)/\Mbargn)  \arrow["\wh{\res}^{\leq 1}", r, hook] \arrow[u, hook]& \RelPic(\ov J_{g,n}(\sigma)^{\leq 1}/\Mbargn^{\leq 1}) \arrow[r, "\wh{\res}^o", "\cong" swap] \arrow[u, hook]& \RelPic(J_{g,n}^\chi/\Mgn) \arrow[u, hook]  
\end{tikzcd}
\end{equation*}
where the horizontal map are the restriction maps over the open substacks 
$$\Mgn\subset \Mbargn^{\leq 1}:=\{X\in \Mbargn: \: X \text{ has at most one node}\} \subset \Mbargn,$$
and the upper arrows are the inclusions given by pull-back along the relative good moduli space morphisms.
Furthermore:
\begin{itemize}
    \item $\ker(\res)\cong \ker(\res^o)$ is generated by the relative boundary divisors of $\ov \J_{g,n}^\chi/\Mbargn$.
    \item The maps $\res$ and $\wh{\res}$ are isomorphisms (or equivalently, $\res^o$ and $\wh{\res}^{\leq 1}$ are isomorphisms) if $\ov \J_{g,n}(\sigma)$ is fine. 
\end{itemize}
\end{enumerate}
\end{theoremalpha}
The Theorem generalizes the results of  Melo-Viviani \cite{MV14} for the Caporaso compactified Jacobian stack $\ov \J_g^{Cap,\chi}$ (resp. space $\ov J_g^{Cap,\chi}$) over $\ov{\M}_g$. Part \eqref{T:thmB1} and the first half of part\eqref{T:thmB2} are a special case of the results of Fringuelli-Viviani \cite{FV1, FV2}, which compute the relative Picard group of the stack of relative $G$-bundles over $\Mgn$, for an arbitrary connected algebraic group $G$.

Moreover, we show the following auxiliary results:
\begin{itemize}
    \item We give a formula for the tautological line bundles on $\J_{g,n}^\chi$ in Proposition \ref{P:taut} in terms of the generators of part \eqref{T:thmB1}.
    %\item We determine the relatively ample cone for $J_{g,n}^\chi/\Mgn$ in Corollary \ref{C:amplecJ}.
    \item We determine the boundary divisors of $\ov \J_{g,n}(\sigma)$ and of $\ov J_{g,n}(\sigma)$ in Proposition \ref{P:bound-div}. 
\end{itemize}

\subsectionnotoc{Sketch of the proof of Theorem \ref{T:thmA}}

We can now give a brief sketch of the proof of Theorem \ref{T:thmA}. First of all, the if part follows from the fact that classical compactified universal Jacobian spaces are polarized over $\Mbargn$ by the Esteves's line bundle, as discussed above. Moreover, the only if part \eqref{T:thmA2} follows from the only if part \eqref{T:thmA1} and the results of \cite[Sec. 5]{FPV3} (which we review in Subsection \ref{Sub:equiv-cUJ}), which describes when two compactified universal Jacobian stacks have associated spaces that are isomorphic over $\Mbargn$. In order to prove the only if part of part \eqref{T:thmA1}, we argue as follows. 
Consider a compactified universal Jacobian space $f(\sigma):\ov J_{g,n}(\sigma)\to \Mbargn$ which is projective over $\Mbargn$ and pick a relatively ample line bundle $\A$. We construct a classical compactified universal Jacobian space $f^E:\ov J_{g,n}^E\to \Mbargn$ such that:
\begin{enumerate}[(a)]
    \item there is an isomorphism  
    $${\ov J_{g,n}(\sigma)}_{|\Mbargn^{\leq 1}}=:\ov J_{g,n}(\sigma)^{\leq 1}\cong (\ov J_{g,n}^E)^{\leq 1}:= {\ov J_{g,n}^E}_{|\Mbargn^{\leq 1}},$$
    where $ \Mbargn^{\leq 1}$ is the open big substack of $\Mbargn$ consisting of stable curves having at most one node; 
    \item the restriction of $\A$ to $\ov J_{g,n}(\sigma)^{\leq 1}$ is isomorphic to the restriction of the Esteves' line bundle $\L_E$ to 
   $(\ov J_{g,n}^E)^{\leq 1}$, up to pull-back of line bundles from $\Mbarg^{\leq 1}$.
\end{enumerate}
We then conclude using that, since $ \Mbargn^{\leq 1}$ is an open big substack of $\Mbargn$, there exists $M\gg 0$ such that 
$$
\ov J_{g,n}(\sigma)=\Proj_{\Mbargn}\bigoplus_{k\in \NN}   f(\sigma)_*(\A^{\otimes kM})\cong 
\Proj_{\Mbargn}\bigoplus_{k\in \NN}   f^E_*(\L_{E}^{\otimes kM}) \cong \ov J_{g,n}^E. 
$$
Let us comment on the proofs of properties (a) and (b).
Property (a) is proved using the results of \cite[Sec. 5]{FPV3} (which we review in Subsection \ref{Sub:equiv-cUJ}), which describes the behavior of compactified universal Jacobians over the generic point of each boundary divisors of $\Mbargn$. Property (b) is proved using Theorem \ref{T:thmB}: the crucial fact is that the the restrictions of the Esteves' line bundles $\L_E$ to the universal Jacobian spaces $J_{g,n}^\chi$ (which we compute in Proposition \ref{P:for-pol}) exhaust, as $E$ varies, the entire relative ample cone of $J_{g,n}^\chi/\Mgn$.

\subsectionnotoc{Open Questions}

While this paper solves completely the problem of the projectivity of universal compactified Jacobian spaces over $\Mbargn$, the method used cannot be applied to the study of the projectivity of non-classical relative compactified Jacobian spaces for other families of curves, expecially the V-compactified Jacobians introduced in \cite{FPV1}. Thus the following problem (raised already in \cite{viviani2023new} and \cite{FPV1}) remains open: 

$\bullet$ Are there examples of relative V-compactified Jacobian spaces for some family $X/S$ of reduced curves (e.g. a fixed curve $X$ over a field $k=\ov k$), which are relatively projective and non-classical?

\vspace{0.2cm}
\subsectionnotoc{Outline of the paper}
   Section \ref{Sec:cuJ} recalls the combinatorial classification of compactified universal Jacobians in terms of V-functions, obtained in \cite{FPV3}. Moreover, in Subsection \ref{Sub:classcJ} we recall the definition of the subclass of classical compactified universal Jacobians and in Subsection \ref{Sub:equiv-cUJ} we recall the results of \cite{FPV3} describing when two compactified universal Jacobian stacks or spaces are isomorphic over $\Mbargn$.

    Section \ref{Sec:Pic-cJ} focuses on relative Picard groups of (compactified) universal Jacobians:  in Subsection \ref{Sub:Pic-J} we compute the relative Picard group of the universal Jacobians  $\J_{g,n}^\chi$ and $J_{g,n}^{\chi}$ over $\mathcal{M}_{g,n}$; in Subsection \ref{Sub:Pic-cJ} we extend these results to universal compactified Jacobian stacks $\overline{\mathcal{J}}_{g,n}(\sigma)$ and spaces $\overline{J}_{g,n}(\sigma)$ over $\overline{\mathcal{M}}_{g,n}$.
    
    In Section \ref{Sec:proj-cJ}, we first compute the class of the natural polarization on classical universal compactified Jacobian spaces (see Theorem \ref{T:polar} and Proposition \ref{P:for-pol})
    and then we prove the main Theorem \ref{T:thmA}.

\subsectionnotoc{Acknowledgements}

This paper owes its existence to Sam Molcho, who first raised (together with Y. Bae) the question of the relative projectivity of classical compactified Jacobian spaces (which was then solved in \cite{FPV1} and \cite{FPV3}), and then kept asking the author to compute the class of the natural polarization on classical compactified universal Jacobians (see Proposition \ref{P:for-pol}).

We thank M. Fava and N. Pagani for the recent collaborations \cite{FPV1},\cite{FPV2},\cite{FPV3}, on which this paper heavily relies, and M. Melo and R. Fringuelli for the collaborations \cite{MV14} and \cite{FV1, FV2}, which play a key role in the proof of Theorem \ref{T:thmB}. We thank N. Pagani for comments on the first version of the paper and, in particular, for spotting an annoying sign mistake (which we have now fixed).

The author is funded by the MUR  ``Excellence Department Project'' MATH@TOV, awarded to the Department of Mathematics, University of Rome Tor Vergata, CUP CUP E83C23000330006, by the  PRIN 2022 ``Moduli Spaces and Birational Geometry''  funded by MUR,  and he is a member of  the GNSAGA section of INdAM.

\section{Compactified universal Jacobians}\label{Sec:cuJ}

Fix a pair of integers $(g,n)\in \NN^2$ such that $2g-2+n>0$. Denote by $\Mbargn$ the proper DM moduli stack of stable $n$-pointed nodal curves $X=(X,p_1,\ldots, p_n)$ of genus $g$, i.e. connected projective nodal curves of arithmetic genus $g$ with $n$ pairwise distinct marked smooth points, and by $\Mgn\subset \Mbargn$ the open substack of $n$-pointed smooth curves of genus $g$.

The aim of this section is to recall the classification of compactified universal Jacobians over  $\Mbargn$, obtained in \cite{FPV3} (based upon the works \cite{FPV1} and \cite{FPV2}). 

First of all, we recall the definition of compactified universal Jacobians.

\begin{definition}\label{D:cJ-Uni}
   A \textbf{compactified universal Jacobian stack} of characteristic $\chi$ over $\Mbargn$ (or of type $(g,n)$) is an open substack $\ov \J_{g,n}^{\chi}$ of the stack $\TF^{\chi}_{g,n}$ parametrizing pairs $(X,I)$ where $X\in \Mbargn$ and $I$ is a rank-$1$ torsion-free sheaf on $X$ (i.e. a coherent sheaf which is pure of dimension one, with support equal to $X$ and having rank $1$ on each generic point of $X$) of Euler characteristic $\chi$, admitting a relative proper good moduli space $F:\ov \J_{g,n}^\chi\xrightarrow{\Xi} \ov J_{g,n}^\chi\xrightarrow{f} \Mbargn$, called a \textbf{compactified universal Jacobian space}. The set of compactified universal Jacobian stacks forms a poset with respect to inclusion.

   A compactified universal Jacobian stack $\ov \J_{g,n}^\chi$ is called \emph{fine} if the good moduli morphism $\Xi:\ov \J_{g,n}^\chi\to \ov J_{g,n}^{\chi}$ is a $\Gm$-gerbe, i.e. if $\ov J_{g,n}^{\chi}=\ov \J_{g,n}^{\chi}\fatslash \Gm$
\end{definition}

Notice that the restriction of any compactified universal Jacobian stack $\ov \J_{g,n}^{\chi}$ over the open substack $\Mgn\subset \Mbargn$ of $n$-pointed smooth curves of genus $g$ is the \textbf{universal Jacobian stack} $\J_{g,n}^\chi$ of characteristic $\chi$, parametrizing pairs $(C,L)$ consisting of a smooth curve $C\in \Mgn$ and a line bundle $L$ on $C$ of Euler characteristic $\chi$. Similarly, the restriction of any compactified universal Jacobian space $\ov J_{g,n}^{\chi}$ over $\Mgn$ is the \textbf{universal Jacobian space} $J_{g,n}^\chi:=\J_{g,n}^\chi\fatslash \Gm$ of characteristic $\chi$.

\vspace{0.1cm}

In order to recall the classification of compactified universal Jacobians obtained in \cite{FPV3}, we need to introduce some combinatorial definitions.

The \textbf{stability domain of type $(g,n)$} is the set $\Dgn$ that parametrizes \emph{half-vine graphs} of type $(g,n)$: for any stable vine graph of type $(g,n)$, i.e. a stable graph of genus $g$ and $n$ legs with two vertices and no loops,  together with the choice of one of its two vertices, we associate the element $(e;h,A)\in \Dgn$ where $e$ is the number of edges, $h$ is the genus of the chosen vertex and $A\subseteq [n]$ is the set of legs rooted at the chosen vertex. 
 
 The set $\Dgn$ comes with two natural structures:
\begin{itemize}
\item the complement $(e;h,A)^c$ of an element $(e;h,A)\in \Dgn$ is defined to be the complementary half-vine graph. Explicitly, $(e;h,A)^c:=(e; g-h-e+1,A^c).$
\item a triangle in $\Dgn$ is a multiset (i.e. repetitions are allowed) $\Delta=[(e_1;h_1,A_1),(e_2;h_2,A_2), (e_3;h_3,A_3)]$ of $3$ elements of $\Dgn$ such that there exists a stable graph of type $(g,n)$ with three vertices $\{v_1,v_2,v_3\}$ and having no loops and at least one edge in between any pair of vertices, and such that each vertex $v_i$ has genus $h_i$, it is joined by $e_i$ edges to the other two vertices, and the legs rooted at $v_i$ are marked by the subset $A_i$.
\end{itemize}

The set $\Dgn$ describes the combinatorial type (or simply the type) of the set $\BCon(X)$ of biconnected subcurves of any $X\in \Mbargn$ (i.e. subcurves $Y$ of $X$ such that both $Y$ and its complementary subcurve $Y^c:=\ov{X-Y}$ are connected) in the following sense: for any $(X,p_i)\in \Mbargn$ there is a function
\begin{equation}\label{E:type}
    \begin{aligned}
      \type=\type_X: \BCon(X) & \longrightarrow \Dgn\\
      Y & \mapsto \type(Y):=(|Y\cap Y^c|; g(Y),\{i\in [n]\: : p_i\in Y\}).
    \end{aligned}
\end{equation}
Note that: 
\begin{itemize}
\item $\type_X(Y^c)=\type_X(Y)^c$ for any $Y\in \BCon(X)$.
\item If $X=Y_1\cup Y_2\cup Y_3$ with $Y_i\in \BCon(X)$ and without pairwise common irreducible components, then $[\type_X(Y_1),\type_X(Y_2),\type_X(Y_3)]$ is a triangle in $\Dgn$.
\end{itemize}

\begin{definition}\label{D:Sigmagn}
   Denote by $\Sigma^\chi_{g,n}$ the set of all functions (called \textbf{V-functions} of type $(g,n)$)
    \begin{align*}
        \sigma:\Dgn&\to \ZZ\\
        (e;h,A)&\mapsto \sigma(e;h,A)
    \end{align*}
    satisfying the following properties:
\begin{enumerate}
\item \label{E:condUni1} for any $(e;h,A)\in \Dgn$, we have 
\begin{equation}\label{E:sumUni}
\sigma(e;h,A)+\sigma((e;h,A)^c)-\chi
\in \{0,1\}.
\end{equation}

An element $(e;h,A)\in \Dgn$ is said to be \emph{$\sigma$-degenerate} if $\sigma(e;h,A)+\sigma((e;h,A)^c)=\chi$,
and \emph{$\sigma$-nondegenerate} otherwise.

\item  \label{E:condUni2} for any triangle $\Delta=[(e_1;h_1,A_1), (e_2;h_2,A_2), (e_3;h_3,A_3)]$ of $\Dgn$, we have that:
\begin{enumerate}
 \item if two among the elements of $\Delta$ are $\sigma$-degenerate, then so is  the third. 
            \item the following holds
            \begin{equation}\label{E:triaUni}
            \sum_{i=1}^{3}\sigma(e_i; h_i,A_i)-\chi
            \in \begin{cases}
                \{1,2\} \textup{ if $(e_i;h_i,A_i)$ is $\sigma$-nondegenerate for all $i$};\\
                \{1\} \textup{ if there exists a unique $i$ such that } \\
                \hspace{2cm} \textup{ $(e_i;h_i,A_i)$ is $\sigma$-degenerate};\\
                \{0\} \textup{ if $(e_i;h_i,A_i)$ is $\sigma$-degenerate for all $i$}.
            \end{cases}
        \end{equation}
\end{enumerate}
We say that $\sigma$ has Euler characteristic $\chi$ and we write $\chi=|\sigma|$. We set 
$$\Sigma_{g,n}:=\coprod_{\chi \in \ZZ}\Sigma_{g,n}^\chi.$$
\end{enumerate}
\end{definition}
 The \emph{degeneracy subset} of $\sigma$ is the collection
\begin{equation}\label{E:degsigma}
\D(\sigma):=\{(e;h,A)\in \Dgn: (e;h,A) \text{ is $\sigma$-degenerate}\}.
\end{equation}
We say that $\sigma$ is \emph{general}  if $\D(\sigma)=\emptyset$.

The set $\Sigma_{g,n}$ of V-functions of type $(g,n)$ come with the following poset structure 
  $$
  \sigma_1\geq \sigma_2 \Longleftrightarrow 
  \begin{sis}
  &|\sigma_1|=|\sigma_2|,\\
  & \sigma_1(e;h,A)\geq \sigma_2(e;h,A) \text{ for any } (e;h,A)\in \Dgn.\\
  \end{sis}
  $$
Note that each $\Sigma_{g,n}^{\chi}$ is a union of connected components of the poset $\Sigma_{g,n}$.

The classification of compactified universal Jacobians over $\Mbargn$ reads as follows. 

\begin{theorem}\label{T:class-cJ}(see \cite[Thm. A]{FPV3})
    There is an anti-isomorphism of posets  
$$
\begin{aligned}
\Sigma_{g,n} & \xrightarrow{\cong} \left\{\text{Compactified universal Jacobian stacks over $\Mbargn$}\right\},\\
\sigma &\mapsto \ov \J_{g,n}(\sigma):=\left\{
\begin{aligned} 
& (X,I) \in  \TF_{g,n}^{|\sigma|}: \: \chi(I_Y)\geq \sigma(\type_X(Y)) \\ 
&  \text{ for any } Y\in \BCon(X) 
\end{aligned} 
\right\},
\end{aligned}
$$
where $\displaystyle I_Y:=\frac{I_{|Y}}{\Tors(I_{|Y})}$ is the torsion-free quotient of the restriction $I_{|Y}$  (it is called the torsion-free restriction of $I$ to $Y$).

Moreover, $\sigma$ is general if and only if $\ov \J_{g,n}(\sigma)$ is fine. 
\end{theorem}
The bijection between the set of \emph{fine} compactified Jacobians and the set of \emph{general} V-functions was shown by Fava in \cite[Thm. A]{fava2024}, building upon the results of \cite{pagani2023stability} and \cite{viviani2023new}. The compactified universal Jacobian space  associated to $\ov \J_{g,n}(\sigma)$ will be denoted by $\ov J_{g,n}(\sigma)$.

Let us end this Section with a result on the singularities of compactified universal Jacobian stacks/spaces, which are deduced from \cite{CMKV1,CMKV2,CMKV3}. Some of these results will be used in Section \ref{Sec:Pic-cJ}.

\begin{proposition}\label{P:singu}
  Let $\sigma \in \Sigma_{g,n}$. Then we have that:
\begin{enumerate}[(i)]
    \item \label{P:singu1} $\ov \J_{g,n}(\sigma)$ is regular.
    \item \label{P:singu2} $\ov J_{g,n}(\sigma)$ has toric (hence rational and Cohen-Macaulay) terminal Gorenstein singularities. 
    \item \label{P:singu3} A closed geometric point $(X,I)$ of $\ov \J_{g,n}(\sigma)$ has image contained in the regular locus of $\ov J_{g,n}(\sigma)$ if and only if the graph $\Gamma_{(X,I)}$, obtained from the dual graph of $X$ by contracting the edges at which $I$ is free, is tree-like, i.e. it becomes a tree after removing all the loops. 
\end{enumerate}
\end{proposition}
\begin{proof}
Part \eqref{P:singu1} follows from the fact that the universal deformation space of a geometric point $(X,I)$ of $\ov \J_{g,n}(\sigma)$ is formally smooth (see \cite[Sec. 3]{CMKV2} or \cite[Thm. 8.3(i)]{FPV1}).  

In order to prove the remaining statements, let $[(X,I)]$ be a geometric point of $\ov J_{g,n}(\sigma)$ and assume (without loss of generality) that $(X,I)$  is a closed geometric point of $\ov \J_{g,n}(\sigma)$ (which is equivalent to the fact that $I$ is $\sigma$-polystable, see \cite[Prop. 6.17]{FPV1}). Then it follows from the arguments of \cite[Thm. A(2)]{CMKV2} (see also \cite[Sec. 8.1]{CMKV3}) that the completion of the local ring of $\ov J_{g,n}(\sigma)$ at $[(X,I)]$ is a power series ring over the cographic toric ring $U(\Gamma_{(X,I)})$ (defined in \cite[Sec. 3]{CMKV3}).
Now, part \eqref{P:singu2} follows from \cite[Thm. 5.5] {CMKV3} while part \eqref{P:singu3} follows from \cite[Prop. 5.7]{CMKV3}.
\end{proof}

\subsection{Classical compactified universal Jacobians}\label{Sub:classcJ}

The aim of this subsection is to recall from \cite[Sec. 4]{FPV3} the description of a subclass of compactified universal Jacobians, called classical compactified universal Jacobians, which generalizes \cite{Kass_2019} and \cite{Melo15} (see also \cite{Melo17} and \cite{melo2019}) from the fine case to the general case.

Denote by $\pi:\Cbargn\to \Mbargn$ the \emph{universal family}  which comes equipped with $n$ pairwise disjoint sections $\{\sigma_i\}_{i=1}^n$. We will denote by $\pi:\Cgn\to\Mgn$ the restriction of the universal family over $\Mgn$. 

\begin{theorem}(Weak Franchetta conjecture)\label{T:RPic-Cbar}
The relative Picard group 
$$\RelPic(\Cbargn)=\RelPic(\Cbargn/\Mbargn):=\frac{\Pic(\Cbargn)}{\pi^*\Pic(\Mbargn)}$$
is  generated by:
\begin{itemize}
    \item the relative dualizing line bundle $\omega_\pi$; 
    \item the image $\Sigma_i:=\Im(\sigma_i)$ of the $i$-th section of $\Cbargn/\Mbargn$ (for all $1\leq i \leq n$);
    \item the boundary line bundles $\{\O(C_{(h,A)})\}_{(h,A)\in \Bgn}$ on $\Cbargn$, where  \begin{equation*}
        \Bgn:=\{(h,A)\: :  0\leq h\leq g, A\subseteq [n], 2h-2+|A|>0, 2g-2h+|A^c|>0\}, 
        \end{equation*}
    and $C_{(h,A)}$ is the divisor of $\Cbargn$ whose generic point is a curve made of two smooth irreducible components $C_1$ and $C_2$ of genera, respectively, $h$ and $g-h$, meeting at a node, and containing the marked points $p_i$ such that, respectively, $i\in A$ or $i\in A^c$, and in such a way that if $(h,n)\neq (\frac{g}{2},0)$ then the tautological point lies on $C_1$;
\end{itemize}
subject to the following relations:
\begin{itemize}
    \item $\O(C_{(h,A)})+\O(C_{(h,A)^c})=0$ where $(h,A)^c:=(g-h,A^c)$; 
    \item if $g=1$ then $\omega_\pi=0$;
    \item if $g=0$ then $\Sigma_1=\ldots=\Sigma_n$ and $\omega_\pi=-2\Sigma_1$.
\end{itemize}
\end{theorem}

In particular, $\RelPic(\Cbargn/\Mbargn)$ is torsion-free unless $n=0$ and $g$ is even, in which case $\O(C_{(g/2,\emptyset)})$ is a $2$-torsion element that generates the torsion subgroup of $\RelPic(\Cbargn)$. 

The above Theorem follows (as explained in \cite[Fact 1]{Kass_2019}) from the identification $\Cbargn\cong \ov\M_{g,n+1}$ and the computation of $\Pic(\Mbargn)$ (see \cite{AC87, AC96} in characteristic zero and \cite{FV0} in arbitrary characteristic).

We now set 
$$
\begin{sis}
%&  \RelPic(\Cbargn):=\Pic(\Cbargn)/\pi^*\Pic(\Mbargn),\\
& \RelPic(\Cbargn)_{\RR}:=\RelPic(\Cbargn)\otimes_{\ZZ}\RR,\\
& \RelPic^{\chi}(\Cbargn)_{\RR}:=\{L\in\RelPic(\Cbargn)_{\RR}\: : \deg_{\pi}(L)=\chi\} \text{ for any } \chi\in \ZZ,\\
& \RelPic_{g,n}^{\ZZ}(\Cbargn)_{\RR}:=\coprod_{\chi \in \ZZ} \RelPic^\chi(\Cbargn)_{\RR},
\end{sis}
$$
where $\deg_{\pi}(L)$ is the $\pi$-relative degree of $L$. Similar definition can be given by replacing $\RR$ with $\QQ$. 

From the  description of $\RelPic(\Cbargn)$ in Theorem \ref{T:RPic-Cbar}, it follows that 
\begin{equation*}
   \RelPic^{\ZZ}(\Cbargn)_{\RR}=\left\{L
    =\beta\omega_{\pi}+\sum_{i=1}^n \alpha_i\Sigma_i+\sum_{(h,A)\in \Bgn}\gamma_{(h,A)}\O(C_{(h,A)})\: : \deg_{\pi}(L)=(2g-2)\beta+\sum_{i=1}^n \alpha_i\in \ZZ\right\}.
\end{equation*}

We have a map
 \begin{equation}\label{E:map-sigma2}
 \begin{aligned}
 \sigma_-:  \RelPic_{g,n}^{\ZZ}(\Cbargn)_{\RR}  & \rightarrow \Sigma_{g,n}\\
   L & \mapsto \sigma_L(e;h,A):=
   \begin{cases}
\lceil \beta(2h-2+1)+\sum_{i\in A} \alpha_i- \gamma_{(h,A)}+\gamma_{(h,A)^c} \rceil & \text{ if } e=1,\\
\lceil \beta(2h-2+e)+\sum_{i\in A} \alpha_i \rceil& \text{ if } e\geq 2,
\end{cases}
   \end{aligned}
 \end{equation}
 such that $|\sigma_L|=\deg_{\pi}(L)$ and whose fibers are the regions of $\RelPic_{g,n}^{\ZZ}(\Cbargn)_{\RR}$ with respect to the following arrangement of hyperplanes 
 \begin{equation}\label{E:arr-hyp}
\begin{aligned} 
\A_{g,n}:=& \bigcup_{\substack{(1;h,A)\in \Dgn \\ k\in \ZZ}}\Bigg\{(2h-2+1)\omega_{\pi}^{\vee}+\sum_{i\in A}\Sigma_i^{\vee}+\O(C_{(h,A)})^{\vee}=k\Bigg\}\\
 &\bigcup_{\substack{(e;h,A)\in \Dgn \text{ with } e\geq 2\\ m\in \ZZ}}\Bigg\{(2h-2+e)\omega_{\pi}^\vee+\sum_{i\in A}\Sigma_i^\vee=m\Bigg\},
\end{aligned} 
\end{equation} 
where $(-)^\vee\in \RelPic_{g,n}(\RR)^\vee$ denotes the functional dual to a certain element. 
 Hence, the map $\sigma_{-}$ induces an order-preserving embedding 
 \begin{equation}\label{E:sigma-pos}
 \begin{aligned}
 \sigma_{-}: \R_{g,n}:=\left\{
   \begin{aligned}
   & \text{Regions of } \RelPic_{g,n}^{\ZZ}(\Cbargn)_{\RR}\\
   &\text{ with respect to } \A_{g,n}
   \end{aligned}\right\} & \hookrightarrow \Sigma_{g,n}\\
   [L] & \mapsto \sigma_{[L]}:=\sigma_L,
   \end{aligned}
 \end{equation}
where $[L]$ denotes the region of $\RelPic_{g,n}^{\ZZ}(\Cbargn)_{\RR}$ containing $L$. Observe that:
\begin{itemize}
\item $\sigma_{[L]}$ is general if and only if $[L]$ is a chamber, i.e. a maximal dimensional region. 
\item Any region of $\R_{g,n}$ contains a rational representative, i.e. it can be written as $[L]$ for some $L\in \RelPic_{g,n}^{\ZZ}(\Cbargn)_{\QQ}$.
\end{itemize}

The V-functions of type $(g,n)$ that are in the image of the map $\sigma_-$  are called \emph{classical}. It is shown in \cite[Thm. 3.9]{fava2024} (see also \cite[Fact 4.5]{FPV3}) that the map $\sigma_-$ is not surjective if and only if $g,n>0$ and $2g+n \geq 8$.

We can now define the classical compactified universal Jacobians over $\Mbargn$.

\begin{definition}\label{D:classcJ}
 The \textbf{classical compactified universal Jacobian stacks} of type $(g,n)$ are the compactified universal Jacobian stacks.
$$
 \ov \J_{g,n}([L]):=\ov \J_{g,n}(\sigma_{[L]})=
   \left\{
\begin{aligned} 
& (X,I) \in  \TF_{g,n}^{\deg_{\pi}(L)}: \: \chi(I_Y)\geq \deg_Y(L_{|X}) \\ 
&  \text{ for any } Y\in \BCon(X)
\end{aligned} 
\right\} \quad \text{ as } [L] \text{ varies in } \R_{g,n}.
$$
Their associated good moduli spaces $\ov J_{g,n}([L])$ \textbf{classical compactified universal Jacobian spaces}. 
\end{definition}

\subsection{Isomorphisms among compactified universal Jacobians}\label{Sub:equiv-cUJ}

The aim of this subsection is to recall from \cite[Sec. 5]{FPV3} when two compactified universal Jacobian stacks or spaces are isomorphic over $\Mbargn$.

Consider the following group 
\begin{equation*}
    \wPR_{g,n}:=\RelPic(\Cbargn)\rtimes (\ZZ/2\ZZ),
\end{equation*}
where $ \ZZ/2\ZZ$ acts on $\RelPic(\Cbargn)$ by mapping a line bundle to its inverse.

The group $\wPR_{g,n}$ acts on the stack $\TF_{g,n}$ in the following way (see \cite[Prop. 5.3]{FPV3})
  \begin{itemize}
    \item an element $L\in \RelPic(\Cbargn)$ acts by sending $\I\in \TF_{g,n}$ to 
    $$L\cdot \I:=\I\otimes L.$$
     \item the generator $\iota$ of $\Z/2\Z$  acts by sending $\I\in \TF_{g,n}$ to 
     $$\iota \cdot \I:=\I^*:={\mathcal Hom}(\I,\omega_{\Cbargn/\Mbargn}).$$
    \end{itemize}
 Moreover, the action  of $\wPR_{g,n}$ permutes the compactified universal Jacobian stacks in such a way that the bijection of Theorem \ref{T:class-cJ} becomes $\wPR_{g,n}$-equivariant with respect to the action of $\wPR_{g,n}$  on $\Sigma_{g,n}$ given by (see \cite[Prop. 5.3]{FPV3}):
\begin{itemize}
    \item an element $L=\beta\omega_{\pi}+\sum_{i=1}^n \alpha_i\Sigma_i+\sum \gamma_{(h,A)}\O(C_{(h,A)})\in \RelPic(\Cbargn)$ acts by 
$$(L\cdot \sigma)(e;h,A):=\sigma(e;h,A)+
    \begin{cases}
 \beta(2h-2+1)+\sum_{i\in A} \alpha_i- \gamma_{(h,A)}+\gamma_{(h,A)^c} & \text{ if } e=1,\\
 \beta(2h-2+e)+\sum_{i\in A} \alpha_i& \text{ if } e\geq 2;
\end{cases}$$
    \item the generator $\iota$ of $\Z/2\Z$ acts by 
    $$(\iota \cdot \sigma)(e;h,A):=
     \begin{cases}
      -\sigma(e;h,A) & \text{ if } (e;h,A)\in \D(\sigma), \\
-\sigma(e;h,A)+1 & \text{ if } (e;h,A)\not \in \D(\sigma).    
     \end{cases}$$
\end{itemize}

Furthermore, the group $\wPR_{g,n}$ acts naturally on $\RelPic^{\ZZ}(\Cbargn)_{\RR}$:  the elements of $\RelPic(\Cbargn)$ act via translations and $\iota$ acts as the inverse. With respect to this action, both the maps \eqref{E:map-sigma2} and the one in Theorem \ref{T:class-cJ} are equivariant. In particular, the action of $\wPR_{g,n}$ on compactified Jacobian stacks preserve the subset of classical compactified Jacobian stacks.

We will also need a decomposition of the poset $\Sigma_{g,n}$ of  V-functions of type $(g,n)$ into a separating and a non-separating part.  First of all, we can partition the stability domain $\Dgn$  into a separating and a non-separating domain
$$
\Dgn=\Dgns\bigsqcup \Dgnns,
$$
where 
\begin{equation*}
  \begin{aligned}
& \Dgns:=\{(1;h,A): (1; h,A)\in \Dgn\} \\
&  \Dgnns:= \{(e;h,A): (e; h,A)\in \Dgn \text{ and } e\geq 2\}.
  \end{aligned}  
\end{equation*}
Note that the above partition is stable under the complement operation $(-)\mapsto (-)^c$ and that each triangle in $\Dgn$ is entirely contained in $\Dgnns$.

Then we can define the poset $\Sigmans$ (resp. $\Sigmas$) of \emph{non-separating} (resp. \emph{separating}) V-function of type $(g,n)$ as the set of functions from $\Dgnns$ (resp. $\Dgns$) to $\Z$ satisfying the same properties as in the above definition of V-functions and endowed with the same order relation. We therefore get an isomorphism of posets (see \cite[Lemma 5.9]{FPV3})
\begin{equation}\label{E:decSigma}
\begin{aligned}
    \Sigma_{g,n} & \xrightarrow{\cong}  {}\Sigmas_{g,n}\times \Sigmans_{g,n}\\
    \sigma&\mapsto (\sigma^s:=\sigma_{|\Dgns},\sigma^{ns}:=\sigma_{|\Dgnns}).
\end{aligned}
\end{equation}

The action of $\wPR_{g,n}$ on $\Sigma_{g,n}$ preserves the above decomposition and its restriction to $\Sigmans_{g,n}$ factors via the quotient (see \cite[Lemma 5.12]{FPV3})
$$
\PR_{g,n}:=\wPR_{g,n}/\langle \O(C_{(h,A)})\rangle_{(h,A)}\cong \RelPic(\Cgn)\rtimes (\ZZ/2\ZZ),
$$
where $\RelPic(\Cgn)$ is the relative Picard group of the universal curve $\C_{g,n}/\M_{g,n}$.

We collect the behavior of classical V-functions with respect to the action of $\wPR_{g,n}$ and the decomposition \eqref{E:decSigma} in the following
\begin{remark}\label{R:prop-class}
\noindent
\begin{enumerate}[(i)]
\item The group $\wPR_{g,n}$ acts naturally on $\RelPic^{\ZZ}(\Cbargn)_{\RR}$:  the elements of $\RelPic(\Cbargn)$ act via translations and $\iota$ acts as the inverse. Moreover, this action preservee the arrangement of hyperplanes $\A_{g,n}$, and hence it acts on the set of regions $\R_{g,n}$ in such a way that 
the map $\sigma_-$ of \eqref{E:sigma-pos} is equivariant (see \cite[Rmk. 5.2]{FPV3}).

 In particular, the action of $\wPR_{g,n}$ on compactified Jacobian stacks preserve the subset of classical compactified Jacobian stacks. 
        \item We say that an element $\sigma\in \Sigmas_{g,n}$ (resp. $\sigma\in \Sigmans_{g,n}$) is classical if $\sigma=\sigma_L^s$ (resp. $\sigma=\sigma_L^{ns}$) for some $L\in \RelPic^{\ZZ}(\Cbargn)_{\RR}$. Then we have that (see \cite[Remark 5.10]{FPV3}):
        \begin{itemize}
        \item any $\sigma\in \Sigmas_{g,n}$ is classical;
        \item any $\sigma\in \Sigma_{g,n}$ is classical if and only if $\sigma^{ns}\in \Sigmans_{g,n}$.
        \end{itemize}
        \item The action of $\PR_{g,n}$ on $\Sigmans_{g,n}$ preserves the classical non-separating V-functions.
    \end{enumerate}
\end{remark}

We can now state the main result of this subsection.

\begin{theorem}\label{T:iso-cJ}
Let $\sigma_1,\sigma_2\in \Sigma_{g,n}$.
 \begin{enumerate}
\item Then $\ov \J_{g,n}(\sigma_1)$ and $\ov \J_{g,n}(\sigma_2)$ are isomorphic over $\Mbargn$ if and only if $\sigma_1$ and $\sigma_2$ lie in the same orbit for the action of $\wPR_{g,n}$ on $\Sigma_{g,n}$.
\item The following conditions are equivalent:
     \begin{enumerate}
         \item 
         the V-functions $\sigma_1^{ns}$ and $\sigma_2^{ns}$ lie in the same orbit for the action of $\PR_{g,n}$ on $\Sigmans_{g,n}$.
         \item  the compactified universal Jacobian stacks $\ov \J_{g,n}(\sigma_1)$ and $\ov \J_{g,n}(\sigma_2)$ are isomorphic over the open locus $\Mbargn^{ns}\subset \Mbargn$ parametrizing curves with no separating nodes.
         \item  the relative good moduli spaces $\ov J_{g,n}(\sigma_1)$ and $\ov J_{g,n}(\sigma_2)$ are isomorphic over $\Mbargn$.
          \item  the relative good moduli spaces $\ov J_{g,n}(\sigma_1)$ and $\ov J_{g,n}(\sigma_2)$ are isomorphic over $\Mbargn^{ns}$.
     \end{enumerate}
 \end{enumerate}
\end{theorem}

It is shown in \cite[Prop. 5.13]{FPV3} that the action of $\wt \PR_{g,n}$ on $\Sigma_{g,n}$ has finitely many orbits. Hence, the above Theorem  implies that there are finitely many isomorphism classes of compactified universal Jacobian stacks (and spaces) over $\Mbargn$.

\section{On the Picard group of (compactified) universal Jacobians}\label{Sec:Pic-cJ}

The aim of this subsection is to compute the Picard group of the universal Jacobian stack/space and its compactifications. 

\subsection{Relative Picard group of the universal Jacobian}\label{Sub:Pic-J}

The aim of this subsection is to compute the relative Picard group of the \emph{universal Jacobian stack} over $\J_{g,n}^{\chi}/\Mgn$, parametrizing pairs $(C,L)$ consisting of an $n$-pointed smooth curve $C$ of genus $g$ together with a line bundle $L$ on $C$ of characteristic $\chi$, and of the \emph{universal Jacobian space} $J_{g,n}^\chi:=\J_{g,n}^\chi \fatslash \Gm/\Mgn$. Note that $J_{g,n}^\chi$ is proper over $\Mgn$ with fiber over $C\in \Mgn$ equal to the Jacobian $J_C^\chi$ of $C$ of characteristic $\chi$.

The universal curve over $\J_{g,n}^\chi$ is $\pi:\Cgn\times_{\Mgn} \J_{g,n}^\chi\to \J_{g,n}^\chi$ and it comes equipped with a universal line bundle $\L$ and with $n$ pairwise disjoint sections $\{\sigma_i\}_{i=1}^n$.

We now recall (see \cite[Chap. XIII, Sec. 4, 5]{Arbarello2011GeometryOA} and the references therein) two standard ways of producing line bundles on the base of a family $\pi:\C\to S$ of smooth curves.
%over an algebraic stack $S$.  

Given a coherent sheaf $\mathcal F$ on $\C$ flat over $S$ (for example a locally free sheaf), we associate a line bundle $d_{\pi}(\mathcal F)=d(\F)$ over the base $S$, called the \emph{determinant of cohomology of $\mathcal F$ with respect to $\pi$} and defined as it follows: choose a complex of locally free sheaves of finite rank $f:K_0\to K_1$ such that $\ker(f)\cong \pi_*(\cal F)$ and $\coker(f)\cong R^1\pi_*(\cal F)$ (this is always possible), and define
\begin{equation}\label{D:detcoho}
d_{\pi}(\cal F):=\det \pi_*(K_0)\otimes \det \pi_*(K_1)^{-1}\in \Pic(\cal X). 
\end{equation}	
The determinant of cohomology is functorial with respect to base change, and it is multiplicative in short exact sequences and compatible with Serre's duality, i.e.:
	\begin{equation}\label{detprop}
	\begin{sis}
	& d_{\pi}(\cal G)= d_{\pi}(\cal F)\otimes d_{\pi}(\cal E) \quad \text{ for any short exact sequence } 0\to\mathcal F\to \mathcal G\to\mathcal E\to 0, \\
	& d_{\pi}(\mathcal F\otimes\omega_\pi)= d_{\pi}(\cal F^{\vee}) \quad \text{ if } \cal F \: \text{  is locally free.} 
	\end{sis}
	\end{equation}
	Moreover, if we have two line bundles $\mathcal F$ and $\mathcal G$ on $S$, we can produce a line bundle $\langle \mathcal F,\mathcal G\rangle_{\pi}=\langle \mathcal F,\mathcal G\rangle$ over $S$, called the \emph{Deligne pairing of $\mathcal F$ and $\mathcal G$ with respect to $\pi$}, which is related to the determinant of cohomology via the following isomorphism:
	\begin{equation}\label{Del-pair}
		\def\arraystretch{1.5}\begin{array}{l}
	\langle \mathcal F,\mathcal G\rangle_{\pi}= d_{\pi}(\mathcal F\otimes\mathcal G)\otimes d_{\pi}(\mathcal F)^{-1}\otimes d_{\pi}(\mathcal G)^{-1}\otimes d_{\pi}(\mathcal O).
	\end{array}
	\end{equation}
	The Deligne pairing is symmetric, bilinear and compatible with sections, i.e.
	\begin{equation}\label{Del-pair-comp}
	\begin{sis}
	& \langle \mathcal F,\mathcal G\rangle_{\pi}= \langle \mathcal G,\mathcal F\rangle_{\pi},\\
	& \langle \mathcal F_1\otimes\mathcal F_2,\mathcal G\rangle_{\pi}=\langle \mathcal F_1,\mathcal G\rangle_{\pi}\otimes\langle \mathcal F_2,\mathcal G\rangle_{\pi},\\
	& \langle\O(\sigma),\cal F\rangle_{\pi}=\sigma^*(\cal F),\text{ where }\sigma\text{ is a section of }\pi.
	\end{sis}
    \end{equation}
    The first Chern classes of the  Deligne pairing is given by 
    \begin{equation}\label{E:c1}
	 c_1(\langle \mathcal F,\mathcal G\rangle_{\pi})=\pi_*(c_1(\mathcal F)\cdot c_1(\mathcal G)).
    \end{equation}	

\begin{remark}
The behavior of the determinant of cohomology and of the Deligne pairing with respect to the relative dualization sheaf $\omega_{\pi}$ (or powers of it) is described in the following relations (which hold for any line bundle $\cal F$ on $\C$ and for any $n\in \Z$):
\begin{equation}\label{E:tau-dual}
\begin{sis}
& \langle \cal F, \omega_{\pi}\rangle_{\pi}= d_{\pi}(\cal F)^{-1}\otimes d_{\pi}(\cal F^{-1})=\langle \cal F, \cal F\rangle_{\pi} \otimes d_{\pi}(\cal F)^{-2}\otimes d_{\pi}(\O)^2, \\
& d_{\pi}(\omega_{\pi}^n\otimes  \cal F)= d_{\pi}(\cal F)^{1-2n}\otimes \langle \cal F, \cal F\rangle_{\pi}^n\otimes d_{\pi}(\cal O)^{6n^2-4n},
\end{sis}
\end{equation}
where the first formula is obtained by applying \eqref{Del-pair} twice, one to the pair $(\cal F, \omega_{\pi})$ and the other to the pair $(\cal F, \cal F^{-1})$, and the second formula is obtained by applying  \eqref{Del-pair} to $\langle \cal F, \omega_{\pi}^n\rangle_{\pi}$ and  using Mumford's formula $d_{\pi}(\omega_{\pi}^n)\cong d_{\pi}(\cal O)^{6n^2-6n+1}$  (see \cite[Chap. XIII, Thm. 7.6]{Arbarello2011GeometryOA}).
\end{remark}
    
In our setting, we will apply the above two constructions to the universal family $\pi:\Cgn\times_{\Mgn} \J_{g,n}^\chi\to \J_{g,n}^\chi$ using the universal line bundle $\L$, the relative dualizing sheaf $\omega_{\pi}$ and  the line bundles $\O(\sigma_i)$ associated to the sections $\sigma_i$: we end up with the following \emph{tautological line bundles}
\begin{equation}\label{E:taut}
\begin{sis}
  & d\left(\L^n\otimes \omega_{\pi}^m(\sum_{i=1}^n\mu_i\sigma_i)\right) \text{ for any } n,m,\mu_i\in \ZZ,\\
  &\left\langle \L^n\otimes \omega_{\pi}^m(\sum_{i=1}^n\mu_i\sigma_i), \L^{n'}\otimes \omega_{\pi}^{m'}(\sum_{i=1}^n\mu_i'\sigma_i)\right\rangle \text{ for any } n,m,\mu_i, n',m',\mu_i'\in \ZZ.
\end{sis}
\end{equation}
It turns out that the tautological line bundles generate the relative Picard group of $\J_{g,n}^\chi/\Mgn$ and more precisely we have the following result.

\begin{theorem}(Fringuelli-Viviani \cite{FV1}) \label{T:RPicJ}
The relative Picard group 
$$\RelPic(\J_{g,n}^\chi)=\RelPic(\J_{g,n}^\chi/\Mgn):=\frac{\Pic(\J_{g,n}^\chi)}{\Pic(\Mgn)}$$ 
is generated by
$$
d(\L), \quad \langle \L, \L \rangle, \quad \xi_i:=\langle \L, \O(\sigma_i)\rangle=\sigma_i^*(\L) \text{ for } 1\leq i \leq n,
$$
subject to the following relations:
\begin{itemize}
    \item if $g=1$ then $\langle \L, \L \rangle\equiv d(\L)^2$; 
    \item if $g=0$ then $\xi:=\xi_1=\ldots \xi_n$, $d(\L)=\chi \xi$ and $\langle \L, \L\rangle= 2(\chi-1)\xi$.
\end{itemize}
\end{theorem}
In particular, $\RelPic(\J_{g,n}^\chi)$ is a free abelian group of rank equal to 
$$
\begin{cases}
 n+2 & \text{ if } g\geq 2;\\
 n+1 & \text{ if } g=1;\\
 1 & \text{ if } g=0.
\end{cases}
$$
The special case $n=0$ (and $g\geq 2$) follows from \cite[Thm. A(1)]{MV14}.
\begin{proof}
This follows for $g\geq 1$ from \cite[Thm. 0.4.1]{FV1}  applied to $T=\Gm$, and for $g=0$ from  \cite[Thm. 0.4.2]{FV1}  applied to $T=\Gm$.
\end{proof}

We now can express the tautological line bundles \eqref{E:taut} in terms of the bases of $\RelPic(\J_{g,n}^\chi/\Mgn)$ provided by the above Theorem \ref{T:RPicJ}. In the following, we will denote by $\equiv$ an equality of line bundles in the relative Picard group of $\J_{g,n}^\chi/\Mgn$ (as opposed to an equality of line bundles on $\J_{g,n}^\chi$) and we will adopt an additive notation for the line bundles.

\begin{proposition}\label{P:taut}
The tautological line bundles \eqref{E:taut} are given by
$$d\left(\L^n\otimes \omega_{\pi}^m(\sum_{i=1}^n\mu_i\sigma_i)\right)\equiv n(1-2m)d(\L)+\left[\binom{n}{2}+nm \right]\langle \L,\L\rangle+n\sum_i \mu_i\xi_i, 
$$
$$\begin{aligned}
\left\langle \L^n\otimes \omega_{\pi}^m(\sum_{i=1}^n\mu_i\sigma_i), \L^{n'}\otimes \omega_{\pi}^{m'}(\sum_{i=1}^n\mu_i'\sigma_i)\right\rangle & \equiv
-2(nm'+mn')d(\L)+(nn'+nm'+mn')\langle \L,\L\rangle+\\
& +\sum_i (\mu_in'+n\mu_i')\xi_i. 
\end{aligned}
$$
\end{proposition}
\begin{proof}
Let us first prove  the second formula:
$$\begin{aligned}
& \left\langle \L^n\otimes \omega_{\pi}^m(\sum_{i=1}^n\mu_i\sigma_i), \L^{n'}\otimes \omega_{\pi}^{m'}(\sum_{i=1}^n\mu_i'\sigma_i)\right\rangle  \equiv & \\
& \equiv nn'\langle \L,\L\rangle +(nm'+mn')\langle \L,\omega_{\pi}\rangle +\sum_i (\mu_in'+n\mu_i')\xi_i\equiv & \text{ by } \eqref{Del-pair-comp} \\
& \equiv nn'\langle \L,\L\rangle +(nm'+mn')(\langle \L,\L\rangle -2d(\L)) +\sum_i (\mu_in'+n\mu_i')\xi_i & \text{ by } \eqref{E:tau-dual}\\
\end{aligned}
$$
which gives the desired formula.

In order to prove the first formula, we first compute
$$
\begin{aligned}
& d\left(\L^n\otimes \omega_{\pi}^m(\sum_{i=1}^n\mu_i\sigma_i)\right)\equiv \langle \L^n,  \omega_{\pi}^m(\sum_{i=1}^n\mu_i\sigma_i)\rangle +d(\L^n)\equiv  & \text{ by } \eqref{Del-pair}, \\
& \equiv nm' \langle \L, \L\rangle -2nm' d(\L) +d(\L^n)  & \text{ by the first formula.} 
\end{aligned}
$$
Then we conclude using the formula 
$$
d(\L^n)\equiv n d(\L)+\binom{n}{2}\langle \L, \L \rangle,$$
which is proved by induction on $|n|$ using the formulas
$$
\begin{sis}
& d(\L^n)\equiv d(\L)+d(\L^{n-1})+(n-1)\langle \L, \L \rangle \text{ for } n\geq 2 \text{ by } \eqref{Del-pair} \text{ applied to } (\L,\L^{n-1}), \\
& d(\L^{-n})\equiv d(\L^{-1})+d(\L^{-(n-1)})+(n-1)\langle \L, \L \rangle \text{ for } n\geq 2 \text{ by } \eqref{Del-pair} \text{ applied to } (\L^{-1},\L^{-(n-1)}), \\
& d(\L^{-1})\equiv -d(\L)+\langle \L, \L \rangle \text{ by } \eqref{Del-pair} \text{ applied to } (\L,\L^{-1}).
\end{sis}
$$
\end{proof}

We now describe the relative Picard group of the universal Jacobian space $J_{g,n}^\chi$.

\begin{theorem}(Fringuelli-Viviani \cite{FV2}) \label{T:RPicJ-rig}
The relative Picard group 
$$\RelPic(J_{g,n}^\chi)=\RelPic(J_{g,n}^\chi/\Mgn):=\frac{\Pic(J_{g,n}^\chi)}{\Pic(\Mgn)}$$ 
is the subgroup of  $\RelPic(\J_{g,n}^\chi/\Mgn)$ given by 
$$\RelPic(J_{g,n}^\chi)=\left\{
(2s-r)d(\L)-s\langle \L,\L\rangle-\sum_i a_i \xi_i \: : 
r\chi+s(2g-2)+\sum_i a_i=0
\right\}.$$
\end{theorem}
In particular, $\RelPic(J_{g,n}^\chi)$ is a free abelian group of rank equal to 
$$
\begin{cases}
 n+1 & \text{ if } g\geq 2;\\
 n & \text{ if } g=1;\\
 0 & \text{ if } g=0.
\end{cases}
$$
The special case $n=0$ (and $g\geq 2$) follows from \cite[Thm. B(1)]{MV14}.
\begin{proof}
This is a special case of the results of \cite[Sec. 5]{FV2} (for an arbitrary reductive group $G$).
For the reader's convenience, we give a direct proof in this special case.

Since $J_{g,n}^\chi=\J_{g,n}^\chi\fatslash \Gm$, the group  $\RelPic(J_{g,n}^\chi)$ is the kernel of the weight homomorphism (see \cite[Sec. 5]{FV2}) 
\begin{equation}\label{E:wei}
\wei^\chi:\RelPic(\J_{g,n}^\chi)\to \Hom(\Gm,\Gm)=\ZZ.
\end{equation}
Then the result follows from Theorem \ref{T:RPicJ} using that the weight function on the generators of $\RelPic(\J_{g,n}^\chi)$ is given by (see \cite[Prop. 4.1.2(1)]{FV1})
$$
\begin{sis}
   &  \wei^{\chi}(d(\L))=\chi, \\
   & \wei^\chi(\langle \L,\L\rangle)=2(\chi+g-1),\\
   & \wei^\chi(\xi_1)=1 \text{ for any } 1\leq i \leq n.
\end{sis}
$$
\end{proof}

As a corollary of the above result, we get a description of the relative ample cone of the projective morphism $J_{g,n}^\chi \to \Mgn$.

\begin{corollary}\label{C:amplecJ}
    The relative ample cone of the morphism $J_{g,n}^\chi \to \Mgn$ is equal to 
$$\Amp(J_{g,n}^\chi/\Mgn)=\left\{
(2s-r)d(\L)-s\langle \L,\L\rangle-\sum_i a_i \xi_i \: : 
r\chi+s(2g-2)+\sum_i a_i=0, r>0
\right\}.$$
\end{corollary}
\begin{proof}
The fiber of $\J_{g,n}^\chi \to \Mgn$ over a geometric point $C\in \Mgn$ is equal to the Jacobian stack $\J_C^\chi$ of $C$ of characteristic $\chi$, which is (non-canonically) isomorphic to $J_C^\chi \times B \Gm$.   The restriction homomorphism (which is well-defined)
$$
\res^{\NS}_C: \RelPic(\J_{g,n}^\chi) \to \Pic(\J_C^\chi)\cong \Pic(J_C^\chi)\times \Pic(B\Gm)\twoheadrightarrow \NS(J_C^\chi)
$$
is given on the generators of $\RelPic(\J_{g,n}^\chi)$ by
(see \cite[Prop. 4.1.2(2)]{FV1}, \cite[Sec. 4]{FV2} and \cite[Sec. 3.3]{Kass_2017}))
$$
\res^{\NS}_C(d(\L))=-[\Theta_C]\footnote{Note that there is a sign mistake in \cite[Claim 2 in the proof of Theorem 7.2]{MV14}, where it is asserted, in our language, that $\res^{\NS}_C(d(\L))=[\Theta_C]$. However, with this sign change, the proof of loc. cit. carries over.},\quad
    \res^{\NS}_C(\langle \L,\L \rangle)=-2[\Theta_C], \quad
    \res^{\NS}_C(\xi_i)=0, $$
    where $[\Theta_C] \in \NS(J_C^\chi)$ is the class of the theta divisors of $J_C^\chi$.
Therefore, we get that 
$$
\res_C^{\NS}((2s-r)d(\L)-s\langle \L,\L\rangle-\sum_i a_i \xi_i)=r [\Theta_C],
$$
which is ample if and only if $r>0$.
\end{proof}

The above Theorem \ref{T:RPicJ-rig} can be interpreted in terms of the Jacobian of the generic curve over $\Mgn$ (whenever it exists), as we now explain \footnote{We thank N. Pagani for suggesting this connection.}.%similarly to what was done in \cite[Sec. 4.1]{Kass_2017}.

\begin{remark}\label{R:}
Assume that $g+n\geq 3$, which is equivalent to the property that $\Mgn$ is generically a scheme, and let $\eta_{g,n}$ be the generic point of $\Mgn$. Consider the generic curve and the generic Jacobian over $\Mgn$:  $$\C_{\eta_{g,n}}:=\Cbargn\times_{\Mgn} \eta_{g,n} \text{ and } J_{\eta_{g,n}}^\chi:=J_{g,n}^\chi\times_{\Mgn} \eta_{g,n}=J^\chi_{\C_{\eta_{g,n}}}.$$
    Then the restriction map induces an isomorphism  (see \cite[Prop. 2.3.2]{FV1})
    \begin{equation}\label{E:iso-res1}
    \RelPic(J_{g,n}^\chi)\xrightarrow{\cong} \Pic(J_{\eta_{g,n}}^\chi).
    \end{equation}
    Now assume that $g\geq 1$, for otherwise the above groups are trivial. Consider the following homomorphism 
    $$
    \rho:\RelPic(J_{g,n}^\chi)\xrightarrow{\cong} \Pic(J_{\eta_{g,n}}^\chi)\to \Pic(J_{\ov \eta_{g,n}}^\chi)\to \NS(J_{\ov \eta_{g,n}}^\chi)\cong \ZZ\langle [\Theta_{g,n}]\rangle,
    $$
    where $\ov \eta_{g,n}$ is a geometric point above $\eta_{g,n}$, $[\Theta_{g,n}]$ is the class in the Neron-Severi of a theta divisor on the geometric generic Jacobian $J_{\ov \eta_{g,n}}^\chi$, and the last equality follows from the arguments of \cite[Appendix B]{BKLV}. From the proof of Corollary \ref{C:amplecJ}, it follows that $\rho$ is given explicitly in terms of the generators described in Theorem \ref{T:RPicJ-rig} as:
   $$\rho((2s-r)d(\L)-s\langle \L,\L\rangle-\sum_i a_i \xi_i )=r[\Theta_{g,n}].$$
    From this formula and Theorem \ref{T:RPicJ-rig}, we deduce two consequences:
    \begin{enumerate}[(i)]
        \item The image of $\rho$, which coincides with the Neron-Severi group of $J_{\eta_{g,n}}^\chi$, is equal to 
        $$\NS(J_{\eta_{g,n}}^\chi)=\Im(\rho)=\begin{sis}
          [\Theta_{g,n}] & \text{ if } n\geq 1,\\
          \frac{2g-2}{\gcd(2g-2,\chi)}[\Theta_{g,n}] & \text{ if } n=0 \: (\text{and hence } g\geq 2).
        \end{sis}$$
        The above formula for $n=0$ was proved in \cite[Sec. 7]{MV14}.
        \item There is an isomorphism 
        $$
        \begin{aligned}
            \iota: \RelPic^0(\C_{g,n})&\xrightarrow{\cong} \ker(\rho)\subseteq \RelPic(J_{g,n}^\chi)\\
            L & \mapsto t_L^*d(\L)-d(\L),
        \end{aligned}
        $$
        where $\RelPic^0(\C_{g,n})$ is the subgroup of the relative Picard group of $\C_{g,n}/\M_{g,n}$ consisting of line bundles having relative degree $0$ and $t_L:J_{g,n}^\chi\to J_{g,n}^\chi$ is the translation by $L$.

        Indeed, Theorem \ref{T:RPic-Cbar} implies that any $L\in \RelPic^0(\C_{g,n})$ can be written as 
        $$L=\omega_{\pi}^m(\sum_i \mu_i \sigma_i) \text{ with } m(2g-2)+\sum_i \mu_i=0,$$
        and then we compute, using Proposition \ref{P:taut}, that 
        $$
        \iota(L)=d(\L\otimes \omega_{\pi}(\sum_i \mu_i \sigma_i))-d(\L)\equiv -2md(\L)+m\langle \L,\L\rangle +\sum_{i} \mu_i \xi_i\in \ker(\rho).
        $$
    \end{enumerate}
    Summing up, using the isomorphism \eqref{E:iso-res1} and the analogous one  (see again \cite[Prop. 2.3.2]{FV1})
    \begin{equation}\label{E:iso-res2}
    \RelPic(C_{g,n})\xrightarrow{\cong} \Pic(\C_{\eta_{g,n}}),
    \end{equation}
we obtain the following exact sequence 
    \begin{equation}\label{E:exa-gen}
    0\to \Pic^0(\C_{\eta_{g,n}}) \xrightarrow{\iota} \Pic(J_{\eta_{g,n}}^\chi)\xrightarrow{\rho} \NS(J_{\eta_{g,n}}^\chi)\to 0.
    \end{equation}
This exact sequence was considered in \cite[Sec. 4.1]{Kass_2017} for $\chi=0$ and $n\geq 1$, in which case 
    $\NS(J_{\eta_{g,n}}^\chi)=\ZZ[\Theta_{g,n}]$ and the sequence splits canonically since $-d(\L)\in \Pic(J_{\eta_{g,n}}^\chi)$ and $\rho(-d(\L))=[\Theta_{g,n}]$. 
\end{remark}

\subsection{Relative Picard group of a compactified universal Jacobian}\label{Sub:Pic-cJ}

The aim of this subsection is to compute the relative Picard group of a compactified universal Jacobian stack $F(\sigma):\ov \J_{g,n}(\sigma)\to \Mbargn$ (resp. space $f(\sigma):\ov J_{g,n}(\sigma)\to \Mbargn$), for any V-function $\sigma\in \Sigma_{g,n}$. 

First of all, note that the universal family $\pi:\Cbargn\times_{\Mbargn} \TF_{g,n}\to \TF_{g,n}$ comes with a universal sheaf $\I$ (whose restriction to $\Cbargn\times_{\Mbargn} \J_{g,n}$ is the universal line bundle $\L$) and with $n$ pairwise disjoint sections $\{\sigma_i\}_{i=1}^n$ which lands in the smooth locus of $\pi$. Hence, we can extend some of the tautological line bundles of \eqref{E:taut} and the line bundles $\xi_i$ of Theorem \ref{T:RPicJ} to the entire stack $\TF_{g,n}$ (and hence on any $\ov \J_{g,n}(\sigma)$) by
\begin{equation}\label{E:taut-TF}
   \begin{sis}
    & d(\I\otimes \omega_{\pi}^{m}(\sum_{i=1}^n \mu_i\sigma_i) \text{ for any } m, \mu_i \in \ZZ,\\
    & \xi_i:=\sigma_i^*(\I). 
   \end{sis} 
\end{equation}
where $d=d_{\pi}$ is, as usual, the determinant of cohomology with respect to the universal family $\pi$. 

We now describe the boundary divisors of $\ov \J_{g,n}(\sigma)$ and $\ov J_{g,n}(\sigma)$, and their relationship trough the morphism $\Xi(\sigma):\ov \J_{g,n}(\sigma)\to \ov J_{g,n}(\sigma)$.

Consider the following boundary prime divisors in $\ov \J_{g,n}(\sigma)$:
\begin{itemize}
    \item $\Delta_{\irr}(\sigma)$ is the prime divisor whose generic point is a pair $(X,L)$ where $X$ is an integral curve of $\Mbargn$ with only one node and $L$ is a line bundle on $X$ of characteristic $|\sigma|$.
    \item If $(1;h,A)\not \in \D(\sigma)$ then $\Delta_{h,A}(\sigma)$ is the prime divisor whose generic point is a pair $(X,L)$ where $X=C_1\cup C_2$ is the generic point of $\Delta_{h,A}$ and $L$ is a line bundle on $X$ such that 
    $$
(\chi(L_{C_1}),\chi(L_{C_2}))=    (\sigma(1;h,A),\sigma(1;g-h,A^c)).$$
\item If $(1;h,A) \in \D(\sigma)$ and $(1;h,A)\neq (1;h,A)^c$ then $\Delta_{h,A}^i(\sigma)$ for $i=1,2$ is the prime divisor whose generic point is a pair $(X,L)$ where $X=C_1\cup C_2$ is the generic point of $\Delta_{h,A}$ and $L$ is a line bundle on $X$ such that 
    $$
(\chi(L_{C_1}),\chi(L_{C_2}))=  
\begin{cases}
    (\sigma(1;h,A)+1,\sigma(1;g-h,A^c)) \text{ if } i=1;\\
    (\sigma(1;h,A),\sigma(1;g-h,A^c)+1) \text{ if } i=2.
\end{cases}
$$
\item If $(1;h,A) \in \D(\sigma)$ and $(1;h,A)= (1;h,A)^c$ (which happens if and only if $n=0$, $g$ and $|\sigma|$ are even and $\sigma(1;g/2,\emptyset)=|\sigma|/2$), then $\Delta_{g/2,\emptyset}^1(\sigma)=\Delta_{g/2,\emptyset}^2(\sigma)$ is the prime divisor whose generic point is a pair $(X,L)$ where $X=C_1\cup C_2$ is the generic point of $\Delta_{h,A}$ and $L$ is a line bundle on $X$ such that 
    $$
(\chi(L_{C_1}),\chi(L_{C_2}))=\left(\frac{|\sigma|}{2}+1,\frac{|\sigma|}{2}\right) \text{ or } \left(\frac{|\sigma|}{2},\frac{|\sigma|}{2}+1\right). 
$$
    \end{itemize}
    
The images of the above divisors in $\ov J_{g,n}(\sigma)$ are equal to the following boundary prime divisors:
\begin{itemize}
\item[$\star$] $\wh{\Delta}_{\irr}:=\Xi(\sigma)(\Delta_{\irr}(\sigma))$.
\item[$\star$] If $(1;h,A)\not \in \D(\sigma)$ then set $\wh{\Delta}_{h,A}(\sigma):=\Xi(\sigma)(\Delta_{h,A}(\sigma))$.
\item[$\star$] If $(1;h,A) \in \D(\sigma)$  then 
$$\wh{\Delta}_{h,A}(\sigma):=\Xi(\sigma)(\Delta^1_{h,A}(\sigma))=\Xi(\sigma)(\Delta^2_{h,A}(\sigma))$$ 
is the prime divisor whose generic point is a pair $(X,I=L_1\oplus L_2)$ where $X=C_1\cup C_2$ is the generic point of $\Delta_{h,A}$ and $L_i$ is a line bundle on $C_i$ such that 
    $$
(\chi(L_{C_1}),\chi(L_{C_2}))=      (\sigma(1;h,A),\sigma(1;g-h,A^c)). 
$$
    \end{itemize}

\begin{proposition}\label{P:bound-div}
    Let $\sigma\in \Sigma_{g,n}^\chi$.
    \begin{enumerate}
        \item \label{P:bound-div1} We have that 
      $$f(\sigma)^*(\Delta_{\irr})=\wh \Delta_{\irr}(\sigma) \text{ and } f(\sigma)^*(\Delta_{(h,A)})=\wh \Delta_{(h,A)} \text{ for any } (h,A)\in \Bgn.$$
      The above are all the boundary irreducible divisors of $\ov J_{g,n}(\sigma)$ and they are all  Cartier. 
        \item \label{P:bound-div2} We have that
        $$F(\sigma)^*(\Delta_{\irr})=\Xi(\sigma)^*(\wh \Delta_{\irr}(\sigma))=\Delta_{\irr}(\sigma),$$ 
        $$
        F(\sigma)^*(\Delta_{h,A})=\Xi(\sigma)^*(\wh \Delta_{h,A}(\sigma))=
        \begin{cases}
           \Delta_{h,A}(\sigma) & \text{ if }   (1;h,A)\not \in \D(\sigma),\\
          \Delta_{h,A}^1(\sigma)+\Delta_{h,A}^2(\sigma) & \text{ if }
           (1;h,A) \in \D(\sigma). 
        \end{cases}
        $$
        The above are all the boundary irreducible divisors of $\ov J_{g,n}(\sigma)$ and they are all  Cartier. 
    \end{enumerate}
\end{proposition}
Note that $(1;h,A)=(1;h,A)^c=(1;g-h,A^c)$ if and only if $n=0$, $g$ is even and $(h,A)=(g/2,\emptyset)$. In this case, we have that $(1;g/2,\emptyset)\in \D(\sigma)$ if and only if $|\sigma|$ is even (see \cite[Prop. 8.5(2)]{FPV3}).
\begin{proof}
This was proved by Melo-Viviani in \cite[Thm. C, Thm. 3.2]{MV14} for the Caporaso compactified Jacobian stack $\ov \J_g^{Cap,\chi}$ (resp. space $\ov J_g^{Cap,\chi}$), in the sense of \cite[Sec. 7]{FPV3}.
We will sketch how to adapt the proof of loc. cit. to our more general setting.

First of all, using \cite[5.2, 5.4]{FPV2}, it is easily checked that we have set-theoretic equalities in all the above formulas. This also shows that all the boundary divisors of $\ov \J_{g,n}(\sigma)$ and $\ov J_{g,n}(\sigma)$ are the ones listed. 

In order to promote the set theoretic-equalities to scheme-theoretic equalities, we can repeast the proofs of \cite[Thm. 3.2, Thm. C]{MV14}, which also shows that all the boundary divisors are Cartier.  
\end{proof}

  In the following, we will denote by $\Mbargn^{\leq 1}$ the open substack of $\Mbargn$ consisting of stable curves with at most one node and, for any $\sigma\in \Sigma_{g,n}$, we set 
  $$
  \ov \J_{g,n}(\sigma)^{\leq 1}:=\ov \J_{g,n}(\sigma)_{|\Mbargn^{\leq 1}} \text{ and } \ov J_{g,n}(\sigma)^{\leq 1}:=\ov J_{g,n}(\sigma)_{|\Mbargn^{\leq 1}}.
  $$

\begin{theorem}\label{T:RelPic-cJ}
Let $\sigma\in \Sigma_{g,n}^\chi$. Consider the following commutative diagram 
\begin{equation}\label{E:diag-res}
\begin{tikzcd}
  \res: \RelPic(\ov \J_{g,n}(\sigma)/\Mbargn)  \arrow["\res^{\leq 1}", r, "\cong" swap] & \RelPic(\ov \J_{g,n}(\sigma)^{\leq 1}/\Mbargn^{\leq 1}) \arrow[r, "\res^o", twoheadrightarrow]& \RelPic(\J_{g,n}^\chi/\Mgn)   \\ 
  \wh{\res}: \RelPic(\ov J_{g,n}(\sigma)/\Mbargn)  \arrow["\wh{\res}^{\leq 1}", r, hook] \arrow[u, hook]& \RelPic(\ov J_{g,n}(\sigma)^{\leq 1}/\Mbargn^{\leq 1}) \arrow[r, "\wh{\res}^o", "\cong" swap] \arrow[u, hook]& \RelPic(J_{g,n}^\chi/\Mgn) \arrow[u, hook]  
\end{tikzcd}
\end{equation}
where the horizontal map are the restriction maps over the open substacks $\Mgn\subset \Mbargn^{\leq 1} \subset \Mbargn$ and the upper arrows are the inclusions given by pull-back along the relative good moduli space morphisms.

Then we have that:
\begin{enumerate}[(i)]
    \item  $\res^{\leq 1}$ is an isomorphism.
    \item $\res^o$ is surjetive and there is a surjection 
    \begin{equation}\label{E:ker-reso}
    \bigoplus_{(1;h,A)\in \D(\sigma)} \frac{\langle \Delta_{h,A}^1(\sigma),  \Delta_{h,A}^1(\sigma)\rangle }{\langle \Delta_{h,A}^1(\sigma)+ \Delta_{h,A}^2(\sigma)\rangle}\twoheadrightarrow \ker(\res^o).
    \end{equation}
    In particular, $\res^o$ is an isomorphism if $\D(\sigma^s)=\emptyset$.
    \item  $\wh{\res}^{o}$  is an isomorphism.
    \item  $\wh{\res}^{\leq 1}$ is injective and it is an isomorphism if $\D(\sigma^{ns})=\emptyset$.  
\end{enumerate}
In particular, if $\sigma$ is general (i.e. if $\D(\sigma)=\emptyset$) then all the horizontal arrows in \eqref{E:diag-res} are isomorphisms.
\end{theorem}
\begin{proof}
We will divide the proof according to the different parts:
\begin{enumerate}[(i)]
    \item The fact that $\res^{\leq 1}$ is an isomorphism follows (see e.g. \cite[Lemma 2.3.1]{FV1}) from the fact that $\ov \J_{g,n}(\sigma)$ is regular by Proposition \ref{P:singu}\eqref{P:singu1} and $\ov \J_{g,n}(\sigma)^{\leq 1}\subset \ov \J_{g,n}(\sigma)$ is a big open substack.
    \item Since $\ov \J_{g,n}(\sigma)^{\leq 1}$ is regular, we have (see e.g. \cite[Prop. 1.9]{PTT}) an exact sequence 
    $$
    \bigoplus_{(1;h,A)\in \D(\sigma)} \frac{\langle \Delta_{h,A}^1(\sigma),  \Delta_{h,A}^1(\sigma)\rangle }{\langle \Delta_{h,A}^1(\sigma)+ \Delta_{h,A}^2(\sigma)\rangle}\to \RelPic(\ov \J_{g,n}(\sigma)^{\leq 1}/\Mbargn^{\leq 1}) \to  \RelPic(\J_{g,n}^\chi/\Mgn)\to 0,$$
    where we have used the description of the pull-back of the boundary divisors via $F(\sigma):\ov \J_{g,n}(\sigma)^{\leq 1}\to \Mbargn^{\leq 1}$ given in Proposition \ref{P:bound-div}\eqref{P:bound-div2}. 
\item The fact that $\wh{\res}^{o}$ is an isomorphism follows (using e.g. \cite[Prop. 1.9]{PTT}) from the fact that $\ov J_{g,n}(\sigma)^{\leq 1}$ is regular by Proposition \ref{P:singu}\eqref{P:singu3} and the description of the pull-back of the boundary divisors via $f(\sigma):\ov J_{g,n}(\sigma)^{\leq 1}\to \Mbargn^{\leq 1}$ given in Proposition \ref{P:bound-div}\eqref{P:bound-div1}. 
    \item The injectivity of $\wh{\res}^{\leq 1}$ follows from the fact that $\ov J_{g,n}(\sigma)^{\leq 1}\subset \ov J_{g,n}(\sigma)$ is a a big open substack.

Let us now assume that $\D(\sigma^{ns})=\emptyset$ and let us show that $\wh \res$ is an isomorphism. 

\un{Claim 1:} There exists $\wh\sigma\in \Sigma_{g,n}$ such that 
$$
\wh{\sigma}\geq \sigma, \D(\wh{\sigma})=\emptyset \text{ and } \wh{\sigma}^{ns}=\sigma^{ns}
$$

Indeed, we have to find an element $\wh \sigma=(\wh \sigma^s,\sigma^{ns})$ such that $\wh \sigma^s\geq \sigma^s$ and $\D(\wh \sigma^s)=\emptyset$. This follows from \cite[Prop. 8.5]{FPV3}, using also that if $n=0$ and $g$ is even then the assumption that $\D(\sigma^{ns})=\emptyset$ implies that  $\chi$ cannot be even because we must have that $\gcd(2g-2,\chi)=1$ by \cite[Thm. 7.1]{FPV3}.

\un{Claim 2:} The inclusion $\ov \J_{g,n}(\wh \sigma)\subseteq \ov \J_{g,n}(\sigma)$ induces an isomorphism 
$$
\ov J_{g,n}(\wh \sigma)\xrightarrow{\cong} \ov J_{g,n}(\sigma)
$$

Indeed, this follows from the proof of the implication $(1)\Rightarrow (3)$ in \cite[Thm. 5.14]{FPV3}.

\vspace{0.1cm}

\noindent Using Claim 1 and Claim 2, we can assume, up to replacing $\sigma$ with $\wh \sigma$, that $\D(\sigma)=\emptyset$. By what proved above, the restriction morphism 
\begin{equation}\label{E:iso-res}
\res: \RelPic(\ov \J_{g,n}(\sigma)/\Mbargn)\xrightarrow{\cong} \RelPic(\J_{g,n}^\chi/\Mgn)
\end{equation}
is an isomorphism. 
Furthermore, since  $\ov \J_{g,n}(\sigma)$ is fine, i.e. $\ov J_{g,n}(\sigma)=\ov \J_{g,n}(\sigma)\fatslash \Gm$, by the assumption that $\D(\sigma)=\emptyset$, then an element of $\RelPic(\ov \J_{g,n}(\sigma)/\Mbargn)$ descends to $\RelPic(\ov J_{g,n}(\sigma)/\Mbargn)$ if and only if it lies in the kernel of the weight homomorphism (defined as in 
\cite[Sec. 5]{FV2}) 
$$\wei^\chi:\RelPic(\ov \J_{g,n}(\sigma))\to \Hom(\Gm,\Gm)=\ZZ,$$
which extends \eqref{E:wei}. Since the same condition characterizes the subgroup $\RelPic(J_{g,n}^\chi/\Mgn)\subseteq \RelPic(\J_{g,n}^{\chi}/\Mgn)$ (see the proof of Theorem \ref{T:RPicJ-rig}), the isomorphism $\res$ of \eqref{E:iso-res} descends to the relative Picard groups of the corresponding relative good moduli space and it shows that $\wh{\res}$ is an isomorphism. 
\end{enumerate}
\end{proof}

\begin{remark}
The surjectivity of $\res$ (and of $\res^o$) can be proved explicitly as follows:  Theorem \ref{T:RPicJ} and Proposition \ref{P:taut} imply that $\RelPic(\J_{g,n}^\chi/\Mgn)$ is generated by $\{d(\L), d(\L\otimes \omega_{\pi})=\langle \L, \L\rangle-d(\L), \{\xi_i\}_{i=1}^n\}$ and each of these generators can be extended to $\RelPic(\ov \J_{g,n}(\sigma)/\Mbargn)$ by \eqref{E:taut-TF}.
\end{remark}

\begin{remark}
For $n=0$ and the Caporaso compactified Jacobian stack $\ov \J_g^{Cap,\chi}$, it follows from \cite[Thm. A(2)]{MV14} that the morphism \eqref{E:ker-reso} is an isomorphism.  The same proof (using a lifting of the test curves of Arbarello-Cornalba \cite{AC87}) works for any compactified universal Jacobian $\ov \J_g(\sigma)$ over $\ov{\M}_g$. It would be interesting to investigate if the result hods true for any compactified universal Jacobian stack over $\Mbargn$.
\end{remark}

\section{On the projectivity of compactified universal Jacobians}\label{Sec:proj-cJ}

The aim of this section is to prove the following Theorem, which describes the compactified universal Jacobian spaces which are projective over $\Mbargn$. 

\begin{theorem}\label{T:cla-proj}
\noindent 
 \begin{enumerate}
     \item \label{T:cla-proj1} A compactified universal Jacobian space is projective over $\Mbargn$ if and only if it is isomorphic over $\Mbargn$ to $\ov J_{g,n}([L])$, for some  $[L]\in \R_{g,n}$.
      \item \label{T:cla-proj2} A compactified universal Jacobian stack has an associated compactified universal Jacobian space which is projective over $\Mbargn$ if and only if it is equal to $\ov \J_{g,n}([L])$ for some (uniquely determined) $[L]\in \R_{g,n}$.
 \end{enumerate}   
\end{theorem}

First of all, we recall that classical compactified Jacobian spaces are projective over $\Mbargn$ (see Theorem \cite[Thm. C]{FPV1}) and we write down an explicit polarization.

\begin{theorem}\label{T:polar}
Fix a region $[L]$ in $\R_{g,n}$ and choose a representative $L$ of the form 
$$
L=\frac{M}{r} \text{ for some } r\in \NN_{>0} \text{ and } M\in \RelPic(\Cbargn).
$$
Consider the line bundle 
\begin{equation}\label{E:lbMr}
\L_{M/r}:=d_{\pi}(\I\otimes p_1^*(\O_{\Cbargn}^{\oplus(r-1)}\oplus M^{-1}))^{-1}
\end{equation}
on $\ov \J_{g,n}([L])$, where $d_{\pi}$ denotes the determinant of cohomology with respect to the universal family $\pi:\Cbargn\times_{\Mbargn} \ov \J_{g,n}([L])\to \ov \J_{g,n}([L])$, $\I$ is the universal sheaf and $p_1$ is the projection onto the first factor. 

Then $\L_{M/r}$ descends to a line bundle on $\ov J_{g,n}([L])$ which is relatively ample over $\Mbargn$.
\end{theorem}
\begin{proof}
Consider the following vector bundle on $\Cbargn$
$$
E=\O_{\Cbargn}^{\oplus(r-1)}\oplus M^{-1},
$$
which has relative slope equal to $\mu_{\pi}(E):=\frac{\deg_{\pi}(E)}{\rk E}=-\deg_{\pi}(L)$.
By construction, we have that 
$$\deg(L_{|X})=-\mu(E_{|X}) \text{ for any } X\in \Mbargn.$$  This implies that $\ov \J_{g,n}([L])$ (resp. $\ov J_{g,n}([L])$) is the Esteves's universal compactified Jacobian stack (resp. space) associated to the vector bundle $E$ (see \cite[Example 6.7(2)]{FPV1}).
Then we conclude by applying \cite[Thm. 7.1]{FPV1}. 
\end{proof}

We now compute the class of the polarization on a classical compactified Jacobian space given in  Theorem \ref{T:polar}, restricted to the universal Jacobian.

\begin{proposition}\label{P:for-pol}
Consider a classical compactified universal Jacobian stack $\ov \J_{g,n}([L])$ and choose a rational representative of $[L]$ of the form
$$
L=\frac{M}{r}, \text{ where } M^{-1}=s\omega_{\pi}+\sum_{i=1}^n a_i \Sigma_i 
+\sum_{(h,A)\in \Bgn}\gamma_{(h,A)}\O(C_{(h,A)})\in \RelPic(\Cbargn)  \text{ and } r>0.
$$
Then the restriction of the line bundle $\L_{M/r}$ of \eqref{E:lbMr} to the relative Picard group of $\J_{g,n}^{\deg_{\pi}(L)}/\Mgn$ is equal to 
\begin{equation}\label{E:for-lbMr}
(\L_{M/r})_{|\J_{g,n}^{\deg_{\pi}(L)}}\equiv (2s-r)d(\L)-s\langle \L,\L\rangle-\sum_i a_i \xi_i.
\end{equation}
\end{proposition}
\begin{proof}
By the definition of $\L_{M/r}$, the functoriality and the additivity of the determinant of cohomology, we have that 
\begin{equation}\label{E:form1}
    (\L_{M/r})_{|\J_{g,n}^{\deg_{\pi}(L)}}=d_{\pi}(\L\otimes p_1^*(\O_{\Cbargn}^{\oplus(r-1)}\oplus M^{-1}))^{-1}=-(r-1)d(\L)-d(\L\otimes \omega_{\pi}^s(\sum_i a_i \sigma_i)).
    \end{equation}
Proposition \ref{P:taut} gives that 
\begin{equation}\label{E:form2}
   d(\L\otimes \omega_{\pi}^s(\sum_i a_i \sigma_i))\equiv (1-2s)d(\L)+s\langle \L,\L\rangle+\sum_{i=1}^na_i \xi_i.
\end{equation}
Combining \eqref{E:form1} and \eqref{E:form2}, we get the assertion.    
\end{proof}

We can now prove the main result of this Section. 

\begin{proof}[Proof of Theorem \ref{T:cla-proj}]
The if part of both Theorems follow from Theorem \ref{T:polar}.

Let us now prove the only if implication of part \eqref{T:cla-proj2}, assuming part \eqref{T:cla-proj1}.
Let $\ov \J_{g,n}(\sigma)$ be a compactified universal Jacobian stack such that $\ov J_{g,n}(\sigma)$ is projective over $\Mbargn$. Then part \eqref{T:cla-proj1} implies that there exists $[L]\in \R_{g,n}$ such that 
$\ov J_{g,n}(\sigma)=\ov J_{g,n}([L])$. By Theorem \ref{T:iso-cJ}, we deduce that $\sigma^{ns}=\PR_{g,n}\cdot [L]$. We now conclude that $\sigma$ is classical by Remark \ref{R:prop-class}.

We now show the only if implication of part \eqref{T:cla-proj1}. Assume that $\ov J_{g,n}(\sigma)$ is projective over $\Mbargn$ and pick a relatively ample line bundle $\A$. Up to taking a power of $\A$, Corollary \ref{C:amplecJ} implies that 
\begin{equation}\label{E:A-expl}
\A_{|J_{g,n}^{|\sigma|}}=(2s-r)d(\L)-s\langle \L,\L\rangle -\sum_{i=1}^n a_i \xi_i \quad \text{ with } r|\sigma|+s(2g-2)+\sum_{i=1}^n a_i=0 \text{ and } r\geq 2.
\end{equation}
Consider the following line bundle on $\Cbargn$
\begin{equation}\label{E:M-expl}
M^{-1}=s\omega_{\pi}+\sum_{i=1}^n a_i \Sigma_i.
\end{equation}

\un{Claim:} There exists a boundary divisor $D:=\sum_{(h,A)\in \Bgn}\gamma_{(h,A)}\O(C_{(h,A)})$ in $\Cbargn$ such that 
$$\sigma^s= \left(\sigma_{\frac{M+D}{r}}\right)^s.$$

Indeed, first of all, using \eqref{E:A-expl} and \eqref{E:M-expl}, we compute  
\begin{equation}\label{E:equ-char}
|\sigma_{\frac{M+D}{r}}|=\deg_{\pi} \frac{M+D}{r}=
\frac{-s(2g-2)-\sum_{i=1}^n a_i}{r}=|\sigma|,
\end{equation}
for any boundary divisor $D$. Consider the function
$$
\begin{aligned}
    \eta: \Bgn & \longrightarrow \QQ,\\
    (h,A) & \mapsto \sigma(1;h,A)+\frac{s(2h-1)+\sum_{i\in A}a_i}{r}.
\end{aligned}
$$
Using \eqref{E:equ-char} and \eqref{E:sumUni}, we compute 
\begin{equation}\label{E:eta-comp}
\eta((h,A)^c)=
\begin{cases}
-\eta(h,A) & \text{ if } (1;h,A)\in \D(\sigma), \\
1-\eta(h,A) & \text{ if } (1;h,A)\not\in \D(\sigma).
\end{cases}
\end{equation}
Using the fact that $\eta$ take values in $\frac{1}{r}\ZZ:=\{q\in \QQ: rq\in \ZZ\}$, the hypothesis that $r\geq 2$ and \eqref{E:eta-comp}, we can choose, for any pair $\{(h,A),(h,A)^c\}$ in $\Bgn$, a pair of integral numbers $\{\gamma_{(h,A)}, \gamma_{(h,A)^c}\}$ in such a way that 
\begin{equation}\label{E:gamma-pro}
\begin{sis}
 &\frac{\gamma_{(h,A)^c}-\gamma_{(h,A)}}{r}=\eta(h,A) & \text{ if } (1;h,A)\in \D(\sigma),\\   
 &\eta(h,A)-1<\frac{\gamma_{(h,A)^c}-\gamma_{(h,A)}}{r}<\eta(h,A)  & \text{ if } (1;h,A)\not\in \D(\sigma).
\end{sis}
\end{equation}
We now compute, using \eqref{E:map-sigma2}, \eqref{E:M-expl} and \eqref{E:gamma-pro}:
$$
\sigma_{\frac{M+D}{r}}(1;h,A)=\left\lceil \frac{-s(2h-1)-\sum_{i\in A}a_i-\gamma_{(h,A)}+\gamma_{(h,A)^c}}{r} \right\rceil=\sigma(1;h,A),
$$
which concludes the proof of the Claim.

\vspace{0.1cm}

Consider now the line bundle 
$$
L:=\frac{M+D}{r}\in \RelPic^{|\sigma|}(\Cbargn)_{\RR}
$$
and its associated classical compactified universal Jacobian stack $\ov \J_{g,n}([L])$ and space $\ov J_{g,n}([L])$.

The Claim implies that 
$$\ov \J_{g,n}(\sigma)^{\leq 1}=\ov \J_{g,n}(\sigma)_{|\Mbarg^{\leq 1}}=\ov \J_{g,n}([L])_{|\Mbargn^{\leq 1}}=\ov \J_{g,n}([L])^{\leq 1},$$
because the restriction of a compactified universal Jacobian stack $\ov \J_{g,n}(\sigma)$ over the generic point of the boundary divisor $\Delta_{h,A}$ of $\Mbargn$ depends only on $\sigma(1;h,A)$. Therefore, we get an isomorphism 
$$\Phi:\ov J_{g,n}(\sigma)^{\leq 1}\xrightarrow{\cong}  \ov J_{g,n}([L])^{\leq 1}$$
over $\Mbargn^{\leq 1}$, which is the identity over $\Mgn$. Consider the induced commutative diagram 
\begin{equation}\label{E:Phi-Pic}
\begin{tikzcd}
  \RelPic(\ov J_{g,n}(\sigma)^{\leq 1}/\Mbargn^{\leq 1}) \arrow["\cong", "\res^o(\sigma)" swap, rd] & & \RelPic(\ov J_{g,n}([L])^{\leq 1}/\Mbargn^{\leq 1})  \arrow["\res^o(L)", ld, "\cong" swap]  \arrow[ll, "\Phi^*" swap, "\cong"] \\
     & \RelPic(J_{g,n}^{|\sigma|}/\Mgn)  &  
\end{tikzcd}
\end{equation}
where the vertical restriction arrows are isomorphisms by Theorem \ref{T:RelPic-cJ}. Using the formulas \eqref{E:for-lbMr} and \eqref{E:A-expl} and the injectivity of $\res^o(\sigma)$, we get that 
\begin{equation}\label{E:A=L}
\res^o(\sigma)(\Phi^*(\L_{(M+D)/r}))=\res^o(\sigma)(\A)\Rightarrow\Phi^*(\L_{(M+D)/r})=\A\in \RelPic(\ov J_{g,n}(\sigma)^{\leq 1}/\Mbargn^{\leq 1}).
\end{equation}

Consider now the morphisms $f(\sigma):\ov J_{g,n}(\sigma)\to \Mbargn$ and $f([L]):\ov J_{g,n}([L])\to \Mbargn$ and their restrictions, respectively, $f(\sigma)^{\leq 1}$ and $f([L])^{\leq 1}$ over $\Mbargn^{\leq 1}$. From \eqref{E:A=L}, we get an isomorphism of graded sheaves of $\O_{\Mbargn^{\leq 1}}$-modules:
\begin{equation}\label{E:iso-grd}
  \bigoplus_{n\in \NN}   f(\sigma)^{\leq 1}_*(\A^{\otimes n})  \xleftarrow[\cong]{\Phi^*}  \bigoplus_{n\in \NN}   f([L])^{\leq 1}_*(\L_{(M+D)/r}^{\otimes n}). 
\end{equation}
Since $\A$ is relatively ample for $f(\sigma)$ by assumption and $\L_{(M+D)/r}$ is relatively ample for $f([L])$ by Theorem \ref{T:polar}, then, by the relative Serre vanishing, there exists $M>>0$ such that 
$$
f(\sigma)_*(\A^{\otimes kM}) \text{ and } f([L])_*(\L_{(M+D)/r}^{\otimes kM}) \text{ are locally free on } \Mbargn \text{ for any } k\geq 0.
$$
Since $j:\Mbargn^{\leq 1}\hookrightarrow \Mbargn$ is a big open subset,  we have the following isomorphisms of graded sheaves of $\O_{\Mbargn}$-modules
\begin{equation}\label{E:j-iso}
    j_*\left(\bigoplus_{k\in \NN}   f(\sigma)^{\leq 1}_*(\A^{\otimes kM})\right)=\bigoplus_{k\in \NN}   f(\sigma)_*(\A^{\otimes kM}) \text{ and }
j_*\left(\bigoplus_{k\in \NN}   f([L]^{\leq 1}_*(\L_{(M+D)/r}^{\otimes kM})\right)=\bigoplus_{k\in \NN}   f([L])_*(\L_{(M+D)/r}^{\otimes kM}).
\end{equation}
Using again that $\A$ is relatively ample for $f(\sigma)$  and $\L_{(M+D)/r}$ is relatively ample for $f([L])$, we have that 
\begin{equation}\label{E:Proj}
    \Proj_{\Mbargn}\bigoplus_{k\in \NN}   f(\sigma)_*(\A^{\otimes kM})\cong \ov J_{g,n}(\sigma) \text{ and }
\Proj_{\Mbargn}\bigoplus_{k\in \NN}   f([L])_*(\L_{(M+D)/r}^{\otimes kM}) \cong \ov J_{g,n}([L]). 
\end{equation}
By combining \eqref{E:iso-grd}, \eqref{E:j-iso} and \eqref{E:Proj}, we get that 
$$\ov J_{g,n}(\sigma)\cong \ov J_{g,n}([L]),$$
which concludes the proof. 
\end{proof}

\begin{corollary}\label{C:proj-coarse}
A compactified universal Jacobian space $\ov J_{g,n}(\sigma)$ has a projective (over $\Spec \Z$) coarse moduli space $|\ov J_{g,n}(\sigma)|$ if and only if it is isomorphic over $\Mbargn$ to $\ov J_{g,n}([L])$, for some  $[L]\in \R_{g,n}$.   
\end{corollary}
\begin{proof}
Since the morphism $f(\sigma):\ov J_{g,n}(\sigma)\to \Mbargn$ is representable, by taking their associated coarse moduli spaces we get the following commutative diagram
    \begin{equation}\label{E:coarse}
\begin{tikzcd}
 \ov J_{g,n}(\sigma) \arrow[r] \arrow["f(\sigma)" swap, d] &   {|\ov J_{g,n}(\sigma)|}  \arrow["|f(\sigma)|", d]  \\
    \Mbargn \arrow[r] & {|\Mbargn|}
\end{tikzcd}
\end{equation}
We now conclude using Theorem \ref{T:cla-proj}\eqref{T:cla-proj1} and the fact that:
\begin{itemize}
    \item $f(\sigma)$ is projective if and only if $|f(\sigma)|$ is projective, since $\QQ$-line bundles on $\ov J_{g,n}(\sigma)/\Mbargn$ descend to $\QQ$-line bundles on $|\ov J_{g,n}(\sigma)|/|\Mbargn|$ and relative ampleness is preserved since the fibers of $|f(\sigma)|$ are finite quotients of the fibers of $f(\sigma)$;
    \item $|f(\sigma)|$ is projective if and only if $|\ov J_{g,n}(\sigma)|$ is projective, since $|\Mbargn|$ is projective over $\Spec \Z$.
\end{itemize}
\end{proof}

    \bibliographystyle{alpha}	
    \bibliography{bibtex}
\end{document}